\documentclass[11 pt,reqno]{amsart}
\usepackage{float}
\usepackage{amsmath}
\usepackage{hyphenat}
\usepackage{etoolbox}
\usepackage{amscd, amsfonts, amssymb, graphicx, color}
\usepackage{url}
\usepackage{cite}
\usepackage{tikz-cd}
\usepackage{hyphenat}
\usepackage{mathtools}
\usepackage[colorlinks=true]{hyperref}
\usepackage{titlesec}
\usepackage{array}
\usepackage{caption}
\usepackage{geometry}
\usepackage[labelfont=bf,textfont=it]{caption}
\titleformat{\section}
{\normalfont\bfseries\large\centering}{\thesection}{1em}{}
\makeatletter
\patchcmd\maketitle
{\uppercasenonmath\shorttitle}
{}
{}{}
\patchcmd\maketitle
{\@nx\MakeUppercase{\the\toks@}}
{\the\toks@}
{}
{}{}
\patchcmd\@settitle{\uppercasenonmath\@title}{\Large}{}{}
\patchcmd\@setauthors
{\MakeUppercase{\authors}}
{\authors}
{}{}
\makeatother
\allowdisplaybreaks
\newtheorem{theorem}{Theorem}[section]

\newtheorem{lemma}{Lemma}[section]
\newtheorem{definition}{Definition}[section]
\newtheorem{remark}{Remark}[section]
\newtheorem{example}{Example}[section]
\newtheorem{corollary}{Corollary}[section]
\newtheorem{Proof of Theorem}{Proof}

\hypersetup{urlcolor=blue, citecolor=red, linkcolor= blue}

\def\C{\mathbb{C}}

\makeatletter
\renewcommand\subsection{\@startsection{subsection}{2}%
	\z@{.7\linespacing\@plus\linespacing}{.5\linespacing}%
	{\normalfont\bfseries}}
\makeatother
	
\begin{document}
		\title[On the estimation of the $q$-numerical radius via Orlicz functions ]
		{On the estimation of the $q$-numerical radius via Orlicz functions}
		
		\author[F. Kittaneh, A. Patra, J. Rani ] { {\large Fuad Kittaneh}$^{1,2}$, {\large Arnab Patra }$^{3}$, {\large Jyoti Rani}$^{4}$}
        
\address{$^{[1]}$ Department of Mathematics, The University of Jordan, Amman, Jordan}
\email{\url{fkitt@ju.edu.jo}}

\address{$^{[2]}$ DEPARTMENT OF MATHEMATICS, KOREA UNIVERSITY, SEOUL 02841, SOUTH KOREA}
\email{\url{fkitt@ju.edu.jo}}
        
		\address{$^{[3]}$ Department of Mathematics, Indian Institute of Technology Bhilai, Chhattisgarh, India 491002}
        \email{\url{arnabp@iitbhilai.ac.in}}

		\address{$^{[4]}$ Department of Mathematics, Indian Institute of Technology Bhilai,  Chhattisgarh, India 491002.
		}
        \email{\url{jyotir@iitbhilai.ac.in}}
		
		\date{\today}
		\maketitle
\section*{Abstract}
This study utilizes Orlicz functions to provide refined lower and upper bounds on the q-numerical radius of an operator acting on a Hilbert space. Additionally, the concept of q-sectorial matrices is introduced and further bounds for the q-numerical radius are established. Our results unify several existing bounds for the q-numerical radius. Suitable examples are provided to supplement the estimations.


		\section{Introduction}
		Let $\mathcal{B(H)}$ 
		denote the $C^*$ algebra of all bounded linear operators acting on $\mathcal{H}$. For $T \in \mathcal{B(H)}$, the numerical range, the numerical radius, and the operator norm are defined, respectively, by 
		\begin{equation*}
			W(T)= \{\langle Tx,x \rangle  : x \in \mathcal{H}, \|x\|=1\},
		\end{equation*} 
		\begin{equation*}
			w(T)= \sup\{| \langle Tx,x \rangle | : x \in \mathcal{H}, \|x\|=1\},			
		\end{equation*}
		and
		\begin{equation*}
			\|T\|= \sup\{| \langle Tx,y \rangle | : x,y \in \mathcal{H}, \|x\|=\|y\|=1\}.
		\end{equation*}	
		It is well known that $w(.)$ defines a norm on $\mathcal{B(H)}$, which is equivalent to  the usual operator norm $\|T\|$ such that for all $T \in \mathcal{B(H)}$,
		\begin{equation}\label{eq2.1}
			\frac{\|T\|}{2} \le w(T) \le \|T\|.
		\end{equation}
		The first inequality is transformed into an equality when $T^2=0$ and the second inequality becomes an equality when the operator $T$ is normal. 
Among the various attempts made to refine inequalities (\ref{eq2.1}), the following inequalities have gained considerable attention and significance:
\begin{equation}\label{eq3.1}
			w(T) \le \frac{1}{2}(\|T\|+\|T^2\|^\frac{1}{2})
		\end{equation}
and        
		\begin{equation}\label{eq3.2}
			\frac{1}{4}\|T^*T+TT^*\| \le w^2(T) \le \frac{1}{2}\|T^*T+TT^*\|, 
		\end{equation}
	where $T \in \mathcal{B(H)}$. Inequalities (\ref{eq3.1}) and (\ref{eq3.2}) refine inequalities (\ref{eq2.1}), see \cite{kittaneh2003numerical,kittaneh2005numerical}. For instance, one may refer to these interesting papers \cite{sababheh2021more,moradi2021more,omar2013numerical,zhang2017properties,vandanjav2012numerical,bhunia2021new,bhunia2021development}.	



    For $T \in \mathcal{B(H)}$ and $|q| \le 1$, the $q$-numerical range $W_q(T)$ and the $q$-numerical radius $w_q(T)$ are defined, respectively, as follows: 
		\begin{equation*}
		W_q(T)=\{ \langle Tx,y \rangle: x,y \in H, \|x\|=\|y\|=1, \langle x,y \rangle=q \}		    
		\end{equation*}
and            
\begin{equation*}
    w_q(T)=\sup\{ |\langle Tx,y \rangle|: x,y \in H, \|x\|=\|y\|=1, \langle x,y \rangle=q \}. 
\end{equation*}
It is worth noting that the $q$-numerical range is a generalized version of the well-known numerical range. For $T \in \mathcal{B(H)}$, the $q$-numerical range is a bounded and convex subset of $\mathbb{C}$, which satisfies $W_q(U^*TU)$ = $W_q(T)$ for any unitary operator $U$ on $\mathcal{H}$. The inclusion relation  $q\sigma(T) \subseteq \overline{W_q(T)}$ is significant in this context, where $\sigma(T) $ denotes the spectrum of $T$. For more such properties, the reader may refer to \cite{gau2021numerical}.

The concept of the $q$-numerical range was introduced by M. Marcus and P. Andersen in 1977 on $n$- dimensional unitary space
\cite{marcus1977constrained}. They demonstrated that the set obtained by rotating the $q$-numerical range about the origin forms an annulus for $q \in \mathbb{C}$ with $|q| \leq 1$. In 1984 \cite{tsing1984constrained}, the convexity of the $q$-numerical range was proved by Nam-Kiu Tsing. In 2002\cite{chien2002davis}, M.T. Chien and H. Nakazato described the boundary of the $q$-numerical range of a square matrix. In 2007 \cite{chien2007q}, M.T. Chien and H. Nakazato examined the $q$-numerical radius of weighted unilateral and bilateral shift operators and obtained the $q$-numerical radii of shift operators with periodic weights. Recently, there have been an increasing interest to estimate the $q$-numerical radius of an operator. Among the recent contributions, $q$-numerical radius inequalities for $2 \times 2$ operator matrices have been given in the paper of Kittaneh-Rashid \cite{kittaneh2025new}. For instance, we refer to \cite{patra2024estimation, fakhri2024q, feki2024joint,rani2025q1}. Some of the important results, which are useful in this context are mentioned below. Let $\mathcal{D}=\{z \in \mathbb{C}: |z|<1\}$ and $\overline{\mathcal{D}}=\{z \in \mathbb{C}: |z|\le1\}$. 

	\begin{lemma}\label{t1.15}\cite[Theorem 1.4]{stankovic2024some}
		Let $T \in\mathcal{B(H)}$ and $q \in \overline{\mathcal{D}}$. Then
		\begin{equation*}
			\frac{|q|}{2}\|T\|\le w_q(T)\le \|T\|.
		\end{equation*}
Also, if $T$ is normal then
\begin{equation*}
    |q|\|T\| \le w_q(T) \le \|T\|.
\end{equation*}
\end{lemma}

\begin{lemma} \cite[Theorem 2.1]{fakhri2024q} \label{it3.3}
			Let $T \in B(H)$, and $q \in (0,1)$, then
			\begin{equation*}\label{ie3.3}
				\frac{q}{2(2-q^2)}\|T\| \le w_q(T) \le \|T\|.
			\end{equation*} 	
		\end{lemma}
		
		\begin{lemma}\cite[Theorem 3.1]{fakhri2024q}\label{it3.4}
			Let $T \in \mathcal{B(H)}$ and $q \in (0,1)$. Then
			\begin{equation*}
				\frac{1}{4}\left(\frac{2}{2-q^2}\right)^2\|T^*T+TT^*\| \le w^2_q(T) \le \frac{(q+2\sqrt{1-q^2})^2}{2}\|T^*T+TT^*\|.
			\end{equation*}	
		\end{lemma}   
 
		\begin{lemma}\cite[Theorem 2.10]{fakhri2024q}\label{it3.5}
			Let $T \in \mathcal{B(H)}$ and $q \in (0,1)$. Then
			\begin{equation*}
 w_q^2(T) \le \frac{q^2}{2}\left( \|T\|+\|T^2\|^\frac{1}{2}\right)^2
   +(1-q^2+q\sqrt{1-q^2})\|T\|^2.   
			\end{equation*}	
		\end{lemma}

\begin{lemma}\cite[Theorem 2.5]{fakhri2024q}\label{it3.6}
			Let $T \in \mathcal{B(H)}$, $T^2=0$ and $q \in (0,1)$. Then
			\begin{equation*}
 w_q^2(T) \le \left(1-\frac{3q^2}{4}+q\sqrt{1-q^2}\right)\|T\|^2.   
			\end{equation*}	
		\end{lemma}        
 Our aim is to get refined and more accurate estimations of the $q$-numerical radius using the concept of Orlicz functions. In the event that $\phi$ is an Orlicz function (see the next section for Orlicz functions), we have proved the following main result in Section 3.

\begin{theorem}\label{nt4.25}
	Let $T \in \mathcal{B}(\mathcal{H})$, $q \in \overline{\mathcal{D}}$ and $\phi$ be an Orlicz function. Then
 \begin{align*}
\phi \left(w_q(T)\right)
  \le & \int_0^1 \phi \left( \left(\sqrt{2}t|q|+2(1-t)\sqrt{1-|q|^2}\right) \| TT^*+T^*T\|^\frac{1}{2}\right)dt\\
  \le & \frac{1}{2}\left(\phi \left( \sqrt{2}|q|\|TT^*+T^*T\|^\frac{1}{2}\right)
  + \phi \left( 2\sqrt{1-|q|^2}\|TT^*+T^*T\|^\frac{1}{2}\right)\right).
 \end{align*}
\end{theorem}
With particular choices of $\phi$ in Theorem \ref{nt4.25}, we derive the following noteworthy result:
\begin{equation}\label{newcor}
		\frac{|q|^2}{4}	\|T^*T+TT^*\| \le w_q^2(T) \le \frac{(2-|q|^2+2\sqrt{2}|q|\sqrt{1-|q|^2})}{2}\|TT^*+T^*T\|
	\end{equation}
[cf. Corollary \ref{eq1.16}]. Relation \eqref{newcor} is capable to provide more accurate estimation of $w_q(T)$ in comparison with  Lemma \ref{t1.15}, Lemma \ref{it3.3} and Lemma \ref{it3.4}. 

Besides that, we establish another significant bound for the $q$-numerical radius of the sum of a finite number of bounded linear operators.
\begin{theorem}
Let $T_i \in \mathcal{B}(\mathcal{H})$ for $i = 1, 2, \dots, n$, $\alpha \in (0,1]$ and $q \in \overline{\mathcal{D}}$. Then for the Orlicz function $\phi$, the following inequality holds
\begin{align*}
 \phi\left(w_q^2\left(\sum_{i=1}^{n}T_i \right) \right)
 \le & \frac{1}{n}\sum_{i=1}^{n}   \phi\left(n^2 \left( |q|^2w^2(T_i)
   +(1-|q|^2+|q|\sqrt{1-|q|^2})\|T_i\|^2\right)\right).  
\end{align*}
\end{theorem}
The aforementioned result is the stronger version of some existing results and for particular choices of $\phi$, it provides refinements over the inequalities stated in Lemma \ref{it3.5} and Lemma \ref{it3.6} [cf. Corollary 3.3]. 
We also focus our study on such matrices whose $q$-numerical ranges lie in certain sector of the complex plane and the class of such type of matrices is an extension of the well-known \textit{sectorial matrices}. Let $M_n$ be the $C^*$-algebra of $n \times n $ matrices and for $A \in M_n$, if $W(A) \subseteq S_\alpha$, where 
		\begin{equation*}
			S_{\alpha}= \{ z \in \mathbb{C} : \Re z >0 , |\Im z| \le \tan(\alpha)(\Re z)\},
		\end{equation*}
 then the collection of all such matrices is known as sectorial matrices. The smallest $\alpha$ will be called the sectorial index of $A$.    In the literature, several researchers have dedicated their work to the properties of the numerical ranges and numerical radii of sectorial matrices. For a detailed review, we refer to the articles \cite{arlinskii2003sectorial,alakhrass2020note,alakhrass2021sectorial,bedrani2021numerical,drury2024numerical, li2025norm,sammour2022geometric}. Yassine Bedrani et al. obtained the following notable relation.
        \begin{lemma} \cite[Proposition 3.1]{bedrani2021numerical}\label{ip1.5}
	Let $A$ be a sectorial matrix. Then
	\begin{equation*}
		\cos \alpha \|A\| \le w(A) \le \|A\|.
	\end{equation*}
\end{lemma}
Lemma \ref{ip1.5} provides a better bound for the left inequality in (\ref{eq2.1}) when $0 \le \alpha < \frac{\pi}{3}.$
 
 Samah Abu Sammour et al. in \cite{sammour2022geometric} have demonstrated two notable advancements as follows: 

  \begin{lemma}\label{img}
      Let $A$ be a sectorial matrix. Then
      \begin{equation*}
          \|\Im(A)\| \le \sin(\alpha)w(A),
      \end{equation*}
  where $\alpha \in [0, \frac{\pi}{2}).$    
  \end{lemma}
  \begin{lemma}\label{kitsec}
 Let $A$ be a sectorial matrix. Then
 \begin{equation*}
     \frac{\|AA^*+A^*A\|}{2(1+\sin^2(\alpha))}\le w^2(A),
 \end{equation*}
  where $\alpha \in [0, \frac{\pi}{2}).$ 
  \end{lemma}
Lemma \ref{img} is a significant improvement on the inequality \( \|\Im(A)\| \le w(A) \) and Lemma \ref{kitsec} enhances the left-hand side inequality in inequalities (\ref{eq3.2}).
  
 In this regard, the concept of $q$-sectorial matrices was introduced in Section 4 which is merely an extension of the class of sectorial matrices. We also establish several significant properties and various inequalities concerning the $q$-numerical radii of $q$-sectorial matrices with the help of Orlicz functions and our results generalize and enhance various existing inequalities.

Notice that Lemma \ref{ip1.5}
 is refinement of \eqref{eq2.1} when $0 \le \alpha < \frac{\pi}{3}$. As an application of one of our main results, for a sectorial matrix $A$, we obtain 
 \begin{equation}\label{neweq}
        \frac{1}{1+\sin(\alpha)}\|A\|\le w(A) \le \|A\|,
    \end{equation}  
which refines \eqref{eq2.1} when $0 \le \alpha < \frac{\pi}{2}$ [cf. Corollary \ref{re3.2eq}]. Moreover, relation \eqref{neweq} is better than Lemma \ref{ip1.5} when $\frac{6\pi}{19} < \alpha < \frac{\pi}{2}$. Furthermore, we generalize Lemma \ref{img} and Lemma \ref{kitsec} for $q$-sectorial matrices and present an enhancement of Lemma \ref{kitsec} under specific conditions. 
In addition, we have also proved the following result which gives the upper bound for the $q$-numerical radius of generalized commutator and anti-commutator matrices [cf. Theorem \ref{th3.6}]. As an application of this result, we derive the following significant bound:
\begin{equation*}
  w(AB \pm BA) \le 2\sqrt{1+\sin^2(\alpha)}\| B\|w(A),
\end{equation*}
where $A$ is a sectorial matrix and $B \in M_n$.
This leads to a sharper bound in comparison to the inequality 
$w(AB \pm BA) \le 2\sqrt{2}\| B\|w(A)$ (see \cite[Theorem 11]{fong1983unitarily}).

The paper is organized as follows. Section 2 contains some preliminary results related to Orlicz functions and some notations, which we use throughout the paper. Further, we obtain some upper bounds for the $q$-numerical radii of sums and products of bounded linear operators in terms of Orlicz functions in Section 3. Moreover, in Section 4, we define a new concept namely, $q$-sectorial matrices which is analogous to the concept of sectorial matrices. We establish several important properties of $q$-sectorial matrices and noteworthy bounds for the $q$-numerical radii of sectorial matrices with respect to Orlicz functions.

		%
		%
		%
		%

\section{Preliminaries}
 Orlicz functions are often used in Orlicz spaces, which extends the concept of the classical Lebesgue spaces.

\begin{definition}\cite{classicalIandII}
    An Orlicz function (non-degenerate) $\phi : [0, \infty) \to [0, \infty)$ is a continuous, convex, and increasing function, and it satisfies the following conditions:
\begin{itemize}
    \item[(i)] $\phi(0) = 0$,
    \item[(ii)] $\phi(t) > 0$ for every $t > 0$,
    \item[(iii)] $\lim_{t \to \infty} \phi(t) = \infty$.
\end{itemize}
\end{definition} 
Also, $\phi(\mu t )\le\mu \phi(t)$, $\alpha \in [0,1]$ is the well-known property of Orlicz functions. An Orlicz function is said to be degenerate if $\phi(t)=0$ for some $t>0$ and an Orlicz function is said to be sub-multiplicative if $\phi(t_1 t_2) \leq \phi(t_1) \phi(t_2)$ for all $t_1, t_2 \geq 0$. In this work, we consider only non-degenerate Orlicz functions. The analysis of numerical radius inequalities with the help of Orlicz functions brings valuable insights to operator theory. Utilizing Orlicz functions, helps establish sharper bounds on the numerical radii of operators, which has important implications in both pure and applied mathematics.
\begin{example}
    Some widely recognized examples of Orlicz functions are:
\begin{itemize}
    \item[(i)] $\phi(t) = t^p$, where $p \geq 1$, which corresponds to the standard $\ell_p$ spaces.
    \item[(ii)] $\phi(t) = e^t - 1$, which defines the exponential Orlicz space.
    \item[(iii)] $\phi(t) = t^p \log(1+t)$ for $p > 0$, represents a mixed power-exponential function.
    \item[(iv)] $\phi(t) = e^{t^2} - 1$, which is a quadratic exponential function.
\end{itemize}
\end{example} 
 An Orlicz function can be expressed by the following integral representation $
\phi(t) = \int_0^t \gamma(u) du,$
where $\gamma$ is a non-decreasing function such that $\gamma(0) = 0$, $\gamma(u) > 0$ for $u > 0$, $\lim\limits_{u \to \infty} \gamma(u) = \infty$, and is known as the kernel of $\phi$. These restrictions exclude the case only when $\phi(t)$ is equivalent to the function $t$. The right inverse $\eta$ of $\gamma$ is defined as $
\eta(v) = \sup \{ u : \gamma(u) \leq v \}, \quad v \geq 0.
$
Then $\eta$ satisfies similar properties as $\gamma$. Using this, the complementary Orlicz function $\psi$ to $\phi$ is defined as $
\psi(s) = \int_0^s \eta(v) dv.$
The pair $(\phi, \psi)$ is called mutually complementary Orlicz functions.
\begin{example}
  A few notable examples of Orlicz functions and their corresponding complementary functions are as follows:
\begin{itemize}
    \item[(i)] Let $\phi(t) = \frac{t^p}{p}$, $p>1$. Then $\psi(s)=\frac{s^q}{q}$, where $\frac{1}{p}+\frac{1}{q}=1$.
    \item[(ii)] $\phi(t) = e^t -t- 1$. Then $\psi(s)=(1+s)\log(1+s)-s$. 
\end{itemize}
\end{example}

The following inequalities are useful in the sequel. First, we recall the well-known Hermite-Hadamard inequality \cite{hadamard1893etude} for a convex function $\phi : I \subseteq \mathbb{R} \to \mathbb{R}$. Suppose that $a, b \in I$ are such that $a < b$. Then the Hermite-Hadamard inequality states that
\begin{equation*}
    \phi\left(\frac{a+b}{2}\right) \leq \int_0^1 \phi\big(ta + (1-t)b\big) dt \leq \frac{\phi(a) + \phi(b)}{2}.
\end{equation*}

\begin{lemma}\cite{classicalIandII}
     (Young's inequality) Let $\phi$ and $\psi$ be two complementary Orlicz functions. Then
\begin{itemize}
    \item[(i)] For $u, v \geq 0$, $uv \leq \phi(u) + \psi(v)$.
    \item[(ii)] For $u \geq 0$, $u\phi(u) = \phi(u) + \psi(\phi(u))$ (equality condition).
\end{itemize}
\end{lemma}

\begin{lemma}\label{lemma multi}\cite{maji2022orlicz}
Let $\phi$ be an Orlicz function and suppose that $a_i \geq 0$ for $i = 1, 2, \dots, n$. Then 
\begin{equation*}
    \phi\left(\frac{1}{n} \sum_{i=1}^{n} a_i\right) \leq \frac{1}{n} \sum_{i=1}^{n} \phi(a_i).
\end{equation*}
\end{lemma}
We establish some notation that will be utilized in the subsequent sections. Every $A \in M_n$ admits the decomposition $A= \Re(A)+i \Im(A)$, where $\Re(A)= \frac{A+A^*}{2}$ and $\Im(A)=\frac{A-A^*}{2i}$ are Hermitian matrices, called the real and imaginary parts of $A$, respectively. The complex conjugate and conjugate transpose of $A$ are denoted by $\overline{A}$ and $A^*$, respectively. 


\section{$q$-Numerical Radius Inequalities via Orlicz Functions}

Consider $q \in \overline{\mathcal{D}}$ and $x \in\mathcal{H}$ with $\|x\|=1$. For any $z \in \mathcal{H}$ satisfying $\langle x,z \rangle=0$ and $\|z\|=1$, let $y=\overline{q}x+\sqrt{1-|q|^2}z$. Then $\|y\|=1$ and $\langle x,y \rangle=q$. Conversely, for any $y \in \mathcal{H}$ with $\|y\|=1$ and $\langle x,y \rangle=q$, set $z=\frac{1}{\sqrt{1-|q|^2}}(y-\overline{q}x)$, resulting in $\langle x,z \rangle=0$ and $\|z\|=1$. Thus, there exists a one-to-one correspondence between such a $z$ and $y$. We start with the proof of Theorem \ref{nt4.25}.

\textbf{Proof of Theorem 1.1:}
Let $x,y \in \mathcal{H}$ with $\|x\|=\|y\|=1$ and $\langle x,y \rangle=q$. Then we can express as $y=\overline{q}x+\sqrt{1-|q|^2}z$ with $\|z\|=1$ and $\langle x,z \rangle=0$. Now,
\begin{align*}
\phi \left(|\langle Tx,y \rangle|\right) =~&\phi \left( |\langle T x,\overline{q}x+\sqrt{1-|q|^2}z\rangle|\right)\\
   \le &~ \phi \left(|q||\langle Tx,x \rangle|+\sqrt{1-|q|^2}|\langle Tx,z \rangle|\right) \quad \quad (\phi ~\text{is increasing})\\
    =&~ \phi \left(\frac{2|q||\langle Tx,x \rangle|+2\sqrt{1-|q|^2}|\langle Tx,z \rangle|}{2}\right).
\end{align*}
Using the Hermite-Hadmard inequality, we have
\begin{align*}
 &\phi \left(|\langle Tx,y \rangle|\right) \\
 \le & \int_0^1 \phi \left(2t|q||\langle Tx,x \rangle|+2(1-t)\sqrt{1-|q|^2}|\langle Tx,z \rangle|\right)dt\\
  = & \int_0^1 \phi \left(2t|q|\left|\langle (\Re(T)+i \Im(T))x,x \rangle \right|+2(1-t)\sqrt{1-|q|^2}\left|\langle (\Re(T)+i \Im(T))x,z \rangle \right|\right)dt\\
 \le & \int_0^1 \phi \left(2t|q||\langle (\Re(T)+i \Im(T))x,x \rangle|+2(1-t)\sqrt{1-|q|^2}\left(|\langle \Re(T)x,z \rangle|+| \langle \Im(T)x,z \rangle|\right)\right)dt\\
  \le & \int_0^1 \phi \Big( 2t|q|\left( \left| \langle \Re(T) x,x \rangle \right|^2+\left| \langle \Im(T) x,x \rangle \right|^2 \right)^\frac{1}{2}\\
  +&2\sqrt{2}(1-t)\sqrt{1-|q|^2}\left(\left| \langle \Re(T) x,z \rangle \right|^2
  +\left| \langle \Im(T) x,z \rangle \right|^2\right)^\frac{1}{2} \Big)dt\\
 \le & \int_0^1 \phi \Big( 2t|q|\left( \| \Re(T) x\|^2+\|\Im(T) x\|^2 \right)^\frac{1}{2}
  +2\sqrt{2}(1-t)\sqrt{1-|q|^2}\left(\| \Re(T) x\|^2
  +\| \Im(T) x\|^2\right)^\frac{1}{2}\Big)dt\\
  \le & \int_0^1 \phi \left( \left(2t|q|+2\sqrt{2}(1-t)\sqrt{1-|q|^2}\right) \| \Re^2(T) +\Im^2(T)\|^\frac{1}{2}\right)dt\\
  = & \int_0^1 \phi \left( \left(\sqrt{2}t|q|+2(1-t)\sqrt{1-|q|^2}\right) \| TT^*+T^*T\|^\frac{1}{2}\right)dt\\
  \le & \int_0^1 \left(t\phi \left( \sqrt{2}|q|\|TT^*+T^*T\|^\frac{1}{2}\right)
  +(1-t) \phi \left( 2\sqrt{1-|q|^2}\|TT^*+T^*T\|^\frac{1}{2}\right)\right)dt \quad(\phi~\text{is convex})\\
  = & \frac{1}{2}\left(\phi \left( \sqrt{2}|q|\|TT^*+T^*T\|^\frac{1}{2}\right)
  + \phi \left( 2\sqrt{1-|q|^2}\|TT^*+T^*T\|^\frac{1}{2}\right)\right).
\end{align*}
By taking the supremum over $x,y \in \mathcal{H}$ with $\|x\| = \|y\|=1$, $\langle x,y \rangle=q$, the required result holds.
\hfill $\qed$

Based on the theorem mentioned above, we derive significant subsequent results as follows:
\begin{remark}
 Take $\phi(t)=t^r$, $r \ge 1$. Then we have
  \begin{align}\label{rresult}
 w_q^r(T) 
  \le & \int_0^1\left( \left(\sqrt{2}t|q|+2\sqrt{1-|q|^2}(1-t)\right)^r\|TT^*+T^*T\|^\frac{r}{2}\right)dt \nonumber\\
  \le & ~2^{\frac{r}{2}-1} \left( |q|^r
  +   2^{\frac{r}{2}}\left(1-|q|^2\right)^\frac{r}{2}\right) \|TT^*+T^*T\|^\frac{r}{2}.
 \end{align}
\end{remark}
\begin{corollary}\label{eq1.16}
Let $T \in \mathcal{B}(\mathcal{H})$ and $q \in \overline{\mathcal{D}}$. Then
    \begin{equation*}
		\frac{|q|^2}{4}	\|T^*T+TT^*\| \le w_q^2(T) \le \frac{(2-|q|^2+2\sqrt{2}|q|\sqrt{1-|q|^2})}{2}\|TT^*+T^*T\|.
	\end{equation*}
\end{corollary}
\begin{proof}
  It is easy to verify that $w_q(\Re(T)) \le w_q(T)$ and $w_q(\Im(T)) \le w_q(T)$. Since $\Re(T)$ and $\Im(T)$ are self-adjoint operators, by Lemma \ref{t1.15},
    we have $|q|\|\Re(T)\| \le w_q(\Re(T))$ and $|q|\|\Im(T)\| \le w_q(\Im(T))$. Therefore,
	\begin{align*}
		|q|^2\|T^*T+TT^*\|=&2|q|^2\|\Re^2(T)+\Im^2(T)\|\\
        \le & 2|q|^2\left(\|\Re^2(T)\|+\|\Im^2(T)\|\right)\\
		= & 2\left(|q|^2\|\Re(T)\|^2+|q|^2\|\Im(T)\|^2\right)\\
        \le & 2 \left(w_q^2(\Re(T))+w_q^2(\Im(T))\right)\\
        \le & 
        4w_q^2(T).
	\end{align*}
	We have
	\begin{equation*}
		\frac{|q|^2}{4}\|TT^*+T^* T\|
		 \le w_q^2(T).
	\end{equation*}
For the other side inequality, taking $r=1$ in relation \eqref{rresult}, yields

\begin{equation*}
    w_q(T) \le \left(|q|+\sqrt{2(1-|q|^2)}\right)\left(\frac{\|TT^*+T^*T\|}{2}\right)^\frac{1}{2}.
\end{equation*}
Thus,
\begin{align*}
 w_q^2(T) \le  (2-|q|^2+2|q|\sqrt{2(1-|q|^2)}) \frac{\|TT^*+T^*T\|}{2}.
\end{align*}
This completes the proof.
\end{proof}
\begin{remark} 
We now present the comparison between Lemma \ref{it3.4} and Corollary \ref{eq1.16}. 
It is evident that $\frac{q^2}{4(2-q^2)^2} \le \frac{q^2}{4}$ for $q \in (0,1)$. 
Furthermore, Figure 1 illustrates that the upper bound for $w_q^2(T)$ mentioned in Corollary \ref{eq1.16} is better than the upper bound stated in Lemma \ref{it3.4} for $q \in (0,1)$.
  \begin{figure}[H]
		\begin{center}
		\includegraphics[scale=0.2]{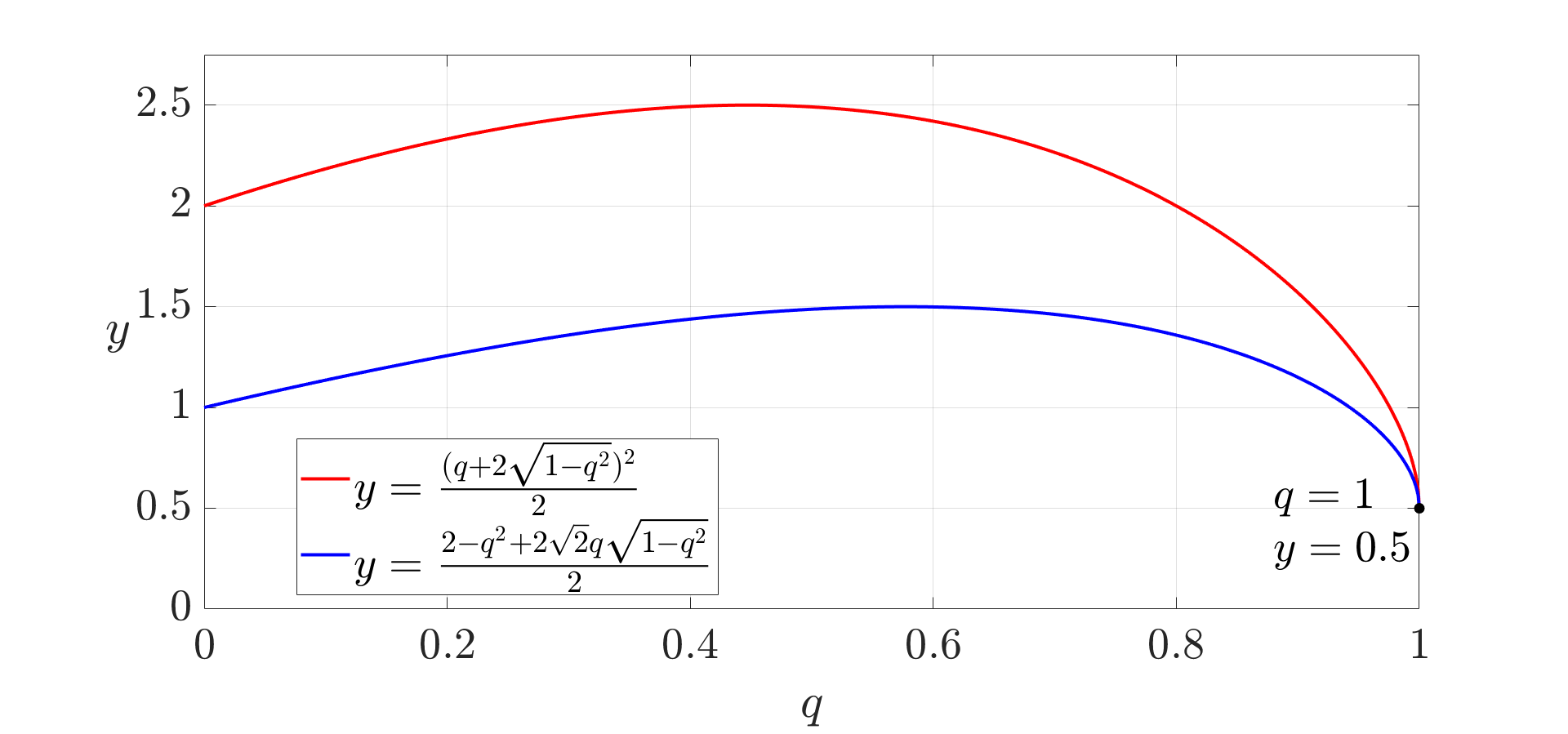}
\caption{Comparison of $\frac{(q+2\sqrt{1-q^2})^2}{2}$ and $\frac{2-q^2+2q\sqrt{2(1-q^2)}}{2}$}        
        \end{center}
		\end{figure}  
 Hence, for $q \in (0,1)$, Corollary \ref{eq1.16} is a refinement of Lemma \ref{it3.4}.
		%

\end{remark}
\begin{remark}
By using the fact $\|T\|^2 \le\|T^*T+TT^*\| $, we have $$\frac{|q|^2}{4}\|T\|^2 \le \frac{|q|^2}{4}\|T^*T+TT^*\|,$$ and hence the left inequality in Corollary \ref{eq1.16} is also a refinement of the left inequality in Lemma \ref{t1.15} and Lemma \ref{it3.3}. Moreover, taking $q=1$ in Corollary \ref{eq1.16}, yields the well known relation \eqref{eq3.2}.      
\end{remark}
	If equality holds in the left inequality of Corollary \ref{eq1.16}, then we obtain the subsequent result.
\begin{corollary}
	Let $T \in \mathcal{B}(\mathcal{H})$, $q \in \overline{\mathcal{D}}$ and $\phi$ be an Orlicz function and $w_q^2(T)=\frac{|q|^2}{4}\|T^*T+TT^*\|$. Then 
\[ \phi \left(|q|^2\|\Re(e^{i\theta}T)\|^2\right)=\phi \left(|q|^2\|\Im(e^{i\theta}T)\|^2\right)=\phi \left(\frac{|q|^2}{4}\|T^*T+TT^*\|\right)\] for all $\theta \in \mathbb{R}$.
\end{corollary}	
\begin{proof}
	By simple calculations, we obtain that
	$$\Re^2(e^{i\theta}T)+\Im^2(e^{i\theta}T)=\frac{TT^*+T^*T}{2},$$ where $\theta \in \mathbb{R}$. 
	Using the convexity of $\phi$, we have
	\begin{align*}
		\phi \left(\frac{|q|^2}{4}\|T^*T+TT^*\|\right)
		=&~\phi \left(\frac{|q|^2}{2}\|	\Re^2(e^{i\theta}T)+\Im^2(e^{i\theta}T)\|\right)\\
		\le &~ \phi \left(\frac{|q|^2\|	\Re^2(e^{i\theta}T)\|}{2}+\frac{|q|^2\|\Im^2(e^{i\theta}T)\|}{2}\right)\\
		\le &~ \frac{1}{2} \left(\phi \left(|q|^2\|	\Re(e^{i\theta}T)\|^2\right)+\phi\left(|q|^2\|	\Im(e^{i\theta}T)\|^2\right)\right)\\
		\le &~ \frac{1}{2}\left(\phi \left(w_q^2(e^{i\theta}T)\right)+\phi \left(w_q^2(e^{i\theta}T)\right)\right)\\
		= &~ \frac{1}{2}\left(\phi \left(w_q^2(T)\right)+\phi \left(w_q^2(T)\right)\right)\\
  =&~\phi \left(w_q^2(T)\right)\\
		=&~\phi \left( \frac{|q|^2}{4}\|T^*T+TT^*\|\right).
	\end{align*}
	Thus, $$\phi \left(|q|^2\|\Re(e^{i\theta}T)\|^2\right)=~\phi \left(|q|^2\|\Im(e^{i\theta}T)\|^2\right)=~\phi \left(w_q^2(T)\right)=~\phi \left(\frac{|q|^2}{4}\|T^*T+TT^*\|\right)$$ for all $\theta \in \mathbb{R}$.
\end{proof}

We will now prove our other main result.

\textbf{Proof of Theorem 1.2:}
Let $x,y \in \mathcal{H}$ with $\|x\| = \|y\|=1$, $\langle x,y \rangle=q$ and $T_i \in \mathcal{B}(\mathcal{H})$ for $i = 1, 2, \dots, n$. Then
\begin{align*}
 \phi\left( \left| \left\langle \left(\sum_{i=1}^{n} T_i \right)x, y\right\rangle\right|^2 \right)  
 = &~\phi\left( \left| \left(\sum_{i=1}^{n} \langle T_i x, y \rangle \right)\right|^2 \right)\\
 \leq&~ \phi\left( \left(\sum_{i=1}^{n} |\langle T_i x, y \rangle| \right)^2 \right)\\
 \leq&~ \phi\left( n \sum_{i=1}^{n} |\langle T_i x, y \rangle|^2 \right)\\
 = & ~\phi\left( \frac{1}{n}\sum_{i=1}^{n} |\langle nT_i x, y \rangle|^2 \right).
\end{align*}
Therefore, Lemma \ref{lemma multi} implies
\begin{equation}\label{rell}
     \phi\left( \left| \left\langle \left(\sum_{i=1}^{n} T_i \right)x, y\right\rangle\right|^2 \right)\leq  \frac{1}{n} \sum_{i=1}^{n} \phi\left( |\langle nT_i x, y \rangle|^2 \right)\leq  \frac{1}{n} \sum_{i=1}^{n} \phi\left(  w_q^2(nT_i) \right).
\end{equation}
Taking $y=\overline{q}x+\sqrt{1-|q|^2}z$ with $\|z\|=1$ and $\langle x,z \rangle=0$, we have
\begin{align*}
 \phi(|\langle T_ix,y \rangle|^2) \le & ~ \phi \left( |\langle T_ix,\overline{q}x+\sqrt{1-|q|^2}z \rangle|^2\right)\\
 \le& ~\phi \left( \left(|q|\langle T_ix,x \rangle|+\sqrt{1-|q|^2}|\langle T_ix,z \rangle|\right)^2\right)\\
=&~ \phi \left( |q|^2 |\langle T_ix,x \rangle|^2+(1-|q|^2)|\langle T_ix,z \rangle|^2+2|q|\sqrt{1-|q|^2}|\langle T_ix,x \rangle||\langle T_ix,z \rangle| \right).
\end{align*}
Using Bessel's inequality, it is easy to verify that $| \langle Tx, z \rangle|^2 \le \|Tx\|^2-|\langle Tx,x \rangle|^2$. Therefore,
\begin{align*}
& \phi(|\langle T_ix,y \rangle|^2)\\
 \le& ~ \phi \bigg( |q|^2 |\langle T_ix,x \rangle|^2+(1-|q|^2)\|T_ix\|^2 
 +  2|q|\sqrt{1-|q|^2}|\langle T_ix,x \rangle|\left(\|T_i x\|^2 - |\langle T_ix, x \rangle|^2\right)^\frac{1}{2} \bigg)\\
\le& ~\phi \bigg( |q|^2 |\langle T_ix,x \rangle|^2+(1-|q|^2)\|T_ix\|^2
 + 2|q|\sqrt{1-|q|^2}\frac{\left(|\langle T_ix,x \rangle|^2+\|T_i x\|^2 - |\langle T_ix, x \rangle|^2\right)}{2} \bigg)\\
\le& ~ \phi \left( |q|^2 |\langle T_ix,x \rangle|^2+(1-|q|^2)\|T_i x\|^2
 +|q|\sqrt{1-|q|^2}\|T_i x\|^2 \right)\\
 \le& ~\phi \left( |q|^2 w^2(T_i)+(1-|q|^2+|q|\sqrt{1-|q|^2})\|T_i\|^2\right).
\end{align*}
Taking the supremum over $x,y \in \mathcal{H}$ with $\|x\|=\|y\|=1$ and $\langle x,y \rangle=    q$, we have
\begin{align}\label{reln}
   \phi\left( w_q^2(T_i) \right)
   \le ~\phi \left( |q|^2w^2(T_i)
   +(1-|q|^2+|q|\sqrt{1-|q|^2})\|T_i\|^2\right). 
\end{align}
Replacing $T_i$ by $nT_i$ in relation \eqref{reln}, we have
\begin{align*}
   \phi\left( w_q^2(nT_i) \right)
   \le&~ \phi \left( |q|^2w^2(nT_i)
   +(1-|q|^2+|q|\sqrt{1-|q|^2})\|nT_i\|^2\right)\\
   =&~\phi\left(n^2 \left( |q|^2w^2(T_i)
   +(1-|q|^2+|q|\sqrt{1-|q|^2})\|T_i\|^2\right)\right).
\end{align*}
By relation \eqref{rell}, we have
\begin{align*}
 \phi\left(\left\langle \left(\sum_{i=1}^{n} T_i \right)x, y\right\rangle^2 \right)
 \le & \frac{1}{n}\sum_{i=1}^{n}   \phi\left(n^2 \left( |q|^2w^2(T_i)
   +(1-|q|^2+|q|\sqrt{1-|q|^2})\|T_i\|^2\right)\right). 
\end{align*}
By taking the supremum over $x,y \in \mathcal{H}$ with $\|x\| = \|y\|=1$, $\langle x,y \rangle=q$, the required result follows.
\hfill $\qed$
  
\begin{remark}
Take $\phi(t)=t^r$. Then we have
\begin{align*}
 w_q^{2r}\left(\sum_{i=1}^{n}T_i \right)
 \le & \frac{1}{n}\sum_{i=1}^{n}  n^{2r} \left( |q|^2w^2(T_i)
   +(1-|q|^2+|q|\sqrt{1-|q|^2})\|T_i\|^2\right)^r.  
\end{align*}
Setting $r=n=1$ in the aforementioned relation, we obtain
\begin{equation}\label{qbound}
 w_q^2(T) \le |q|^2w^2(T)
   +(1-|q|^2+|q|\sqrt{1-|q|^2})\|T\|^2,   
\end{equation} 
which is mentioned in \cite{patra2024estimation}. Moreover, using relation \eqref{eq3.1}, we see that \eqref{qbound} provides a refinement of Lemma \ref{it3.5}.
\end{remark}

In the following remark, we present a refinement of Lemma \ref{it3.6}. 
\begin{remark}
 Let $T\in \mathcal{B(H)}$, $q \in \overline{\mathcal{D}}$ and $T^2=0$. Then relation \eqref{qbound} gives a more accurate bound of $w_q^2(T)$.
In fact,
 A. Abu-Omar and F. Kittaneh \cite[Theorem 2.4]{abu2015upper} derived the following notable bound:
 \begin{equation}\label{amer}
  w^2(T) \le \frac{w(T^2)}{2}+\frac{\|T^*T+TT^*\|}{4}.
\end{equation}
 In particular, if $T^2 = 0,$ then using the inequality \eqref{amer} and the equality $\|T^*T+TT^*\|=\|T\|^2$ \cite[Proposition 1]{kittaneh2005numerical} in the inequality \eqref{qbound}, we obtain
 \begin{align*}
     w_q^2(T)\le &|q|^2w^2(T)
    +(1-|q|^2+|q|\sqrt{1-|q|^2})\|T\|^2\\
    \le & |q|^2\left(\frac{w(T^2)}{2}+\frac{\|T^*T+TT^*\|}{4}\right)
   +(1-|q|^2+|q|\sqrt{1-|q|^2})\|T\|^2\\
 \le &\frac{|q|^2}{4}\|T\|^2
   +(1-|q|^2+|q|\sqrt{1-|q|^2})\|T\|^2\\
= &
\left(1-\frac{3}{4}|q|^2+|q|\sqrt{1-|q|^2}\right)\|T\|^2.   
 \end{align*}
Therefore, relation \eqref{qbound} is a refinement of Lemma \ref{it3.6} \cite[Theorem 2.5]{fakhri2024q}. 
\end{remark}

Next, we derive some results on the $q$-numerical radius of the product of operators via an Orlicz function $ \phi $.

\begin{theorem}\label{thmproduct}
 Let $T,R,S \in \mathcal{B}(\mathcal{H})$, $q \in \overline{\mathcal{D}}$ and $\phi$ be any Orlicz function. Then  
 \begin{align*}
 &\phi \left(w_q(TRS)\right)   \\
 \le& \int_0^1 \phi \Bigg(t|q| \|R\|\||S|^2+|T^*|^2\|+2(1-t)\Big(\sqrt{1-|q|^2}
   +\sqrt{2|q|\sqrt{1-|q|^2}}\Big)\|R\|\||S|^2\|^\frac{1}{2}\||T^*|^2\|^\frac{1}{2} \Bigg)dt\\
   \le & \frac{1}{2}\Bigg( \phi \left(|q|\|R\| \||S|^2+|T^*|^2\|\right)+\phi \Big(2\Big(\sqrt{1-|q|^2}
   +\sqrt{2|q|\sqrt{1-|q|^2}}\Big)\|R\|\||S|^2\|^\frac{1}{2} \||T^*|^2\|^\frac{1}{2}\Big)\Bigg).
 \end{align*}
\end{theorem}
\begin{proof}
 For each $x,y \in \mathcal{H}$ with $\|x\| = \|y\|=1$, $\langle x,y \rangle=q$, we can express $y=\overline{q}x+\sqrt{1-|q|^2}z$ where $\|z\|=1$ and $\langle x,z \rangle=0$. Therefore,
 \begin{align*}
    &\phi \left( \left| \langle TRS x,y \rangle\right|\right)
    =~
    \phi\left( \left| \langle RS x,T^*y \rangle\right|\right)
  \le ~  \phi \left( \|R\| \|Sx\| \|T^*y\| \right)\\
  =& ~ \phi \left( \|R\| \langle Sx, Sx \rangle^\frac{1}{2} \langle T^*y, T^*y \rangle^\frac{1}{2} \right)\\
  =&~ \phi \left( \|R\| \langle Sx, Sx \rangle^\frac{1}{2} \langle T^*(\overline{q}x+\sqrt{1-|q|^2}z), T^*(\overline{q}x+\sqrt{1-|q|^2}z) \rangle^\frac{1}{2} \right)\\
  \le&~ \phi \left( \|R\| \langle S^*Sx, x \rangle^\frac{1}{2} \left(|q|^2\langle TT^*x,x\rangle+(1-|q|^2)\langle TT^*z,z\rangle+2|q|\sqrt{1-|q|^2}|\langle TT^*x,z\rangle|\right)^\frac{1}{2} \right)\\
   \le&~ \phi \left( \|R\| \langle S^*Sx, x \rangle^\frac{1}{2}\left( |q|\langle TT^*x,x\rangle^\frac{1}{2}+\sqrt{1-|q|^2}\langle TT^*z,z\rangle^\frac{1}{2}+\sqrt{2|q|\sqrt{1-|q|^2}}|\langle TT^*x,z\rangle|^\frac{1}{2}\right) \right)\\
   \le& ~\phi \Big(|q|\|R\|\langle S^*Sx, x \rangle^\frac{1}{2}\langle TT^*x,x\rangle^\frac{1}{2}+\|R\|\langle S^*Sx, x \rangle^\frac{1}{2}\Big(\sqrt{1-|q|^2}\langle TT^*z,z\rangle^\frac{1}{2}\\
   +&\sqrt{2|q|\sqrt{1-|q|^2}}|\langle TT^*x,z\rangle|^\frac{1}{2} \Big)\Big)\\
    \le& ~\phi \Big(|q|\|R\|\left( \frac{\langle S^*Sx, x \rangle+\langle TT^*x,x\rangle}{2}\right)+\|R\|\langle S^*Sx, x \rangle^\frac{1}{2}\Big(\sqrt{1-|q|^2}\langle TT^*z,z\rangle^\frac{1}{2}\\
   +&\sqrt{2|q|\sqrt{1-|q|^2}}|\langle TT^*x,z\rangle|^\frac{1}{2} \Big)\Big)\\
    =&~ \phi \Big(|q|\|R\|\left( \frac{\langle S^*Sx, x \rangle+\langle TT^*x,x\rangle}{2}\right)+\frac{1}{2}\Big(2\|R\|\langle S^*Sx, x \rangle^\frac{1}{2}\Big(\sqrt{1-|q|^2}\langle TT^*z,z\rangle^\frac{1}{2}\\
   +&~\sqrt{2|q|\sqrt{1-|q|^2}}|\langle TT^*x,z\rangle|^\frac{1}{2}\Big) \Big)\Big).
  \end{align*}
  Using the Hermite-Hadmard inequality, we have 
  \begin{align*}
  &\phi \left(\langle TRS x,y \rangle \right)\\
    \le& \int_0^1 \phi \Big(t|q|\|R\|\left( \langle S^*Sx, x \rangle+\langle TT^*x,x\rangle\right)+2(1-t)\|R\|\langle S^*Sx, x \rangle^\frac{1}{2}\Big(\sqrt{1-|q|^2}\langle TT^*z,z\rangle^\frac{1}{2}\\
   +&\sqrt{2|q|\sqrt{1-|q|^2}}|\langle TT^*x,z\rangle|^\frac{1}{2} \Big)\Big)dt\\
   \le& \int_0^1 \phi \Bigg(t|q|\|R\|\||S|^2+|T^*|^2\|+2(1-t)\Big(\sqrt{1-|q|^2}|
   +\sqrt{2|q|\sqrt{1-|q|^2}}\Big)\|R\|\||S|^2\|^\frac{1}{2}\||T^*|^2\|^\frac{1}{2} \Bigg)dt\\
   \le& \int_0^1 \Bigg(t \phi \left(|q| \|R\|\||S|^2+|T^*|^2\|\right)+(1-t)\phi \Big(2\Big(\sqrt{1-|q|^2}
   +\sqrt{2|q|\sqrt{1-|q|^2}}\Big)\|R\|\||S|^2\|^\frac{1}{2} \||T^*|^2\|^\frac{1}{2}\Big)\Bigg)dt\\
   =& \frac{1}{2}\Bigg( \phi \left(|q|\|R\| \||S|^2+|T^*|^2\|\right)+\phi \Big(2\Big(\sqrt{1-|q|^2}
   +\sqrt{2|q|\sqrt{1-|q|^2}}\Big)\|R\|\||S|^2\|^\frac{1}{2} \||T^*|^2\|^\frac{1}{2}\Big)\Bigg).
 \end{align*}
 This completes the proof.
\end{proof}

Based upon the aforementioned theorem and considering various choices of $\phi$, we establish the following corollary.
\begin{corollary}
\begin{itemize}
Let $T,R,S \in \mathcal{B}(\mathcal{H})$ and $q \in \overline{\mathcal{D}}$. Then 
    \item [(i)]\begin{align*}
 w_q^r(TRS)
 \le& \int_0^1  \Bigg(t|q| \|R\|\||S|^2+|T^*|^2\|+2(1-t)\Big(\sqrt{1-|q|^2}
   +\sqrt{2|q|\sqrt{1-|q|^2}}\Big)\|R\|\||S|^2\|^\frac{1}{2}\||T^*|^2\|^\frac{1}{2} \Bigg)^rdt\\
   \le & \frac{\|R\|^r}{2}\Bigg( |q|^r \||S|^2+|T^*|^2\|^r+2^r\Big(\sqrt{1-|q|^2}
   +\sqrt{2|q|\sqrt{1-|q|^2}}\Big)^r\||S|^2\|^\frac{r}{2} \||T^*|^2\|^\frac{r}{2}\Bigg),
 \end{align*}
 \item[(ii)] 
$
 w_q(TRS)
   \le  \frac{\|R\|}{2}\Bigg( |q| \||S|^2+|T^*|^2\|+2\Big(\sqrt{1-|q|^2}
   +\sqrt{2|q|\sqrt{1-|q|^2}}\Big)\||S|^2\|^\frac{1}{2} \||T^*|^2\|^\frac{1}{2}\Bigg), 
$    
\item[(iii)] 
$
 w_q(T^2)
   \le  \frac{1}{2}\Bigg(|q|\||T|^2+|T^*|^2\|+2\Big(\sqrt{1-|q|^2}
   +\sqrt{2|q|\sqrt{1-|q|^2}}\Big)\||T|^2\|\Bigg). 
$    
\end{itemize}

\begin{itemize}
\item [(iv )] Moreover, if $R$ is a contraction operator and $\phi$ be any Orlicz function, then we have
\begin{align*}
 \phi \left(w_q(TRS)\right) 
 \le& \int_0^1 \|R\|\phi \Bigg(t|q| \||S|^2+|T^*|^2\|+2(1-t)\Big(\sqrt{1-|q|^2}
   +\sqrt{2|q|\sqrt{1-|q|^2}}\Big)\||S|^2\|^\frac{1}{2}\||T^*|^2\|^\frac{1}{2} \Bigg)dt\\
   \le & \frac{\|R\|}{2}\Bigg( \phi \left(|q| \||S|^2+|T^*|^2\|\right)+\phi \Big(2\Big(\sqrt{1-|q|^2}|
   +\sqrt{2|q|\sqrt{1-|q|^2}}\Big)\||S|^2\|^\frac{1}{2} \||T^*|^2\|^\frac{1}{2}\Big)\Bigg),
 \end{align*}
\end{itemize}    
\end{corollary}
\begin{proof}
 By choosing $ \phi(t) = t^r $ and $ \phi(t) = t $ in Theorem \ref{thmproduct}, part (i)  and part (ii) follow, respectively. Moreover, part (iii) follows by setting $ R = I $ and $S = T $ in part (ii). Since $ R $ is a contraction operator, we have $ \|R\| \leq 1 $. Using the condition \( \phi(\mu t) \leq \mu \phi(t) \) when $\mu \in [0,1]$, part (iv) holds. 
\end{proof}

\section{$q$-Numerical ranges and radii of $q$-sectorial matrices}	
In this section, we define the notion of $q$-sectorial matrices, which is analogous to sectorial matrices. A matrix $A \in M_n$ is said to be $q$-sectorial if $W_q(A) \subseteq S_{\alpha}$ for some $\alpha \in \mathclose{[}0,\frac{\pi}{2}\mathopen{)}$, where 
		\begin{equation*}			
				S_{\alpha}= \{ z \in \mathbb{C} : \Re z >0 , |\Im z| \le \tan(\alpha)(\Re z)\},
		\end{equation*}
and $\Re z$, $\Im z$ denote the real and imaginary parts of $z$, respectively. 
		The smallest $\alpha$ will be called the $q$-sectorial index of $A$.
  \begin{figure}[H]
      \centering
      \includegraphics[scale=0.5]{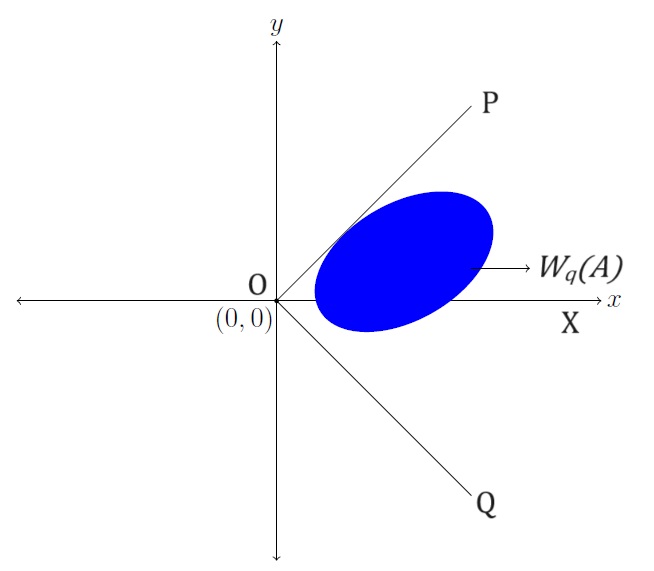}
      \caption{$q$-Sectorial matrix $A$ with $q$-sectorial index $\angle POX= \angle QOX= \alpha$.}
  \end{figure}      
			
			
	 Geometrically, the $q$-numerical range of a $q$-sectorial matrix lies entirely in a cone in the right half complex plane $(i.e., \{z \in \mathbb{C}: \Re z>0\})$ with vertex at the origin and half-angle $\alpha$. In case if $q=1$, then the concept of $q$-sectorial matrices reduces to the concept of sectorial matrices. Here, we denote the set of all $q$-sectorial matrices by $\prod_{q,\alpha}^n$. Moreover, $\prod_{1,\alpha}^n$ denotes the class of all sectorial matrices.

 Due to the relation $q \sigma(A) \subseteq W_q(A)$, the key defining property of a $q$-sectorial matrix for $q \in (0,1)$ is that the spectrum of a $q$-sectorial matrix lies in a sector of the complex plane, which is a region bounded by two rays emanating from the origin. 
 
 Moreover, a $q$-sectorial matrix $A$ satisfies the following properties.
 \begin{lemma}\label{nlma}
 Let $A \in \prod_{q,\alpha}^n$. Then 
 \begin{itemize}
     \item [(i)] $A^T \in \prod_{q,\alpha}^n$, where $q \in \overline{\mathcal{D}}$,
     \item [(ii)] $\overline{A} \in \prod_{q,\alpha}^n$, where $q \in (0,1)$,
     \item [(iii)] $A^* \in \prod_{q,\alpha}^n$, where $q \in (0,1)$,
     \item [(iv)] $\lambda_1A+\lambda_2B \in \prod_{q,\alpha}^n$, where $q \in \overline{\mathcal{D}}$ and $\lambda_1, \lambda_2>0$. Moreover, it is easy to verify that the class $\prod_{q,\alpha}^n$ does not form a subspace of $M_n$,
     \item[(v)]For $q \in (0,1)$, a Hermitian matrix $A$ with eigenvalues $\lambda_1 \ge \lambda_2\ge...\ge \lambda_n$ is a $q$-sectorial matrix if and only if $q(\lambda_1+\lambda_n)>|\lambda_1-\lambda_n|$.  
\end{itemize}  
 \end{lemma}
 \begin{proof}   
 Let $x,y \in \mathbb{C}^n$ with $\|x\|=\|y\|=1$ and $\langle x,y \rangle=q$. Consider
\begin{equation}\label{equalities}
    \langle A^T x,y \rangle =y^*A^T x
    =(y^*A^T x)^T
    =x^TA \overline{y}
    =y_1^* A x_1=\langle Ax_1,y_1 \rangle,
\end{equation}
where $x_1= \overline{y}$ and $y_1=\overline{x}$. Thus, $W_q(A)=W_q(A^T)$. Therefore, property (i) holds. As similar to relation \eqref{equalities}, we have $\langle A^*x,y \rangle=\langle \overline{A} x',y' \rangle,$ where $x'=\overline{y}$ and $y'=\overline{x}$. This implies $W_q(A^*)=W_q(\overline{A})$, and using Proposition 3.1(g) \cite{gau2021numerical}, we have $W_q(A^*)=W_q(A)^*=W_q(\overline{A})$, $q \in (0,1)$. Therefore, properties (ii) and (iii) hold. It is easy to verify that $\Re\langle(\lambda_1A+\lambda_2B)x,y \rangle$. Using the triangle inequality and $\frac{a+b}{c+d} \le \max \{\frac{a}{c}, \frac{b}{d}\}$ for $a,b,c,d>0$, we obtain
\begin{align*}
 \frac{\left| \Im\langle(\lambda_1A+\lambda_2B)x,y \rangle \right|}{\Re\langle(\lambda_1A+\lambda_2B)x,y \rangle} \le \max \left\{ \frac{\lambda_1\left|\Im\langle Ax,y \rangle\right|}{\lambda_1\Re\langle Ax,y \rangle}, \frac{\lambda_2\left|\Im\langle Bx,y \rangle\right|}{\lambda_2\Re\langle Bx,y \rangle }  \right\}\le \tan(\alpha). 
\end{align*}
 Hence, property (iv) follows. Now, we will prove part (v). Since $A$ is an $n \times n$ Hermitian matrix with eigenvalues $\lambda_1\ge\lambda_2 \ge...\ge \lambda_n$, using Theorem 3.5 \cite[p. 384]{gau2021numerical}, we have $W_q(A)$ is an elliptical disc defined by
\[W_q(A)= \left\lbrace  (x,y): \frac{(x-\frac{q}{2}(\lambda_1+\lambda_n))^2}{\frac{1}{4}(\lambda_1-\lambda_n)^2}+\frac{y^2}{\frac{1}{4}(1-q^2)(\lambda_1-\lambda_n)^2}\le1 \right\rbrace,\] 
where $q \in (0,1)$. Let $A$ be a $q$-sectorial matrix. If  $q(\lambda_1+\lambda_n)\le |\lambda_1-\lambda_n|$, then the ellipse intersects the $y$-axis or touches the $y$-axis, and there is no $\alpha \in[0,\frac{\pi}{2})$ such that $W_q(A) \subseteq S_\alpha$. Consequently, $A$ is not $q$-sectorial, which is a contradiction. This implies that $q(\lambda_1+\lambda_n)> |\lambda_1-\lambda_n|$. Conversely, if $q(\lambda_1+\lambda_n)> |\lambda_1-\lambda_n|$, then $W_q(A)$ is always located in the right half plane and does not touch the $y$-axis. Thus, there exists an $\alpha \in[0,\frac{\pi}{2})$ such that $W_q(A) \subseteq S_\alpha$. Therefore, $A$ is $q$-sectorial. This completes the proof.   
 \end{proof}
It is a well-established fact that a positive semidefinite matrix is always a sectorial matrix with index $\alpha=0$. However, in the following example, we can see that a positive semidefinite matrix is not necessarily a $q$-sectorial matrix.
\begin{example}
    Let $q \in (0,1)$, $A_1=\begin{pmatrix}
        \frac{5}{2} & \frac{-1}{2} \\
        \frac{-1}{2} & \frac{5}{2}
    \end{pmatrix}$ and $A_2=\begin{pmatrix}
        4 & -3 \\
        -3 & 4
    \end{pmatrix}$. The eigenvalues of $A_1$ are $2,3$ and the eigenvalues of $A_2$ are $1,7$. Then, using Theorem 3.5 \cite[p.384]{gau2021numerical} $W_q(A_1)$ and $W_q(A_2)$ are elliptical discs as follows:
    \begin{equation*}
        W_q(A_1)=\left\lbrace (x,y) \in \C^2 : \frac{(x-\frac{5q}{2})^2}{\frac{1}{4}}+\frac{y^2}{\frac{1}{4}(1-q^2)} \le 1 \right\rbrace
 \end{equation*}    
 and
 \begin{equation*}
     W_q(A_2)=\left\lbrace (x,y) \in \C^2 : \frac{(x-4q)^2}{9}+\frac{y^2}{9(1-q^2)}\le 1 \right\rbrace.   
 \end{equation*}
      
 \begin{figure}[H]
\begin{center}
 \includegraphics[scale=0.2]{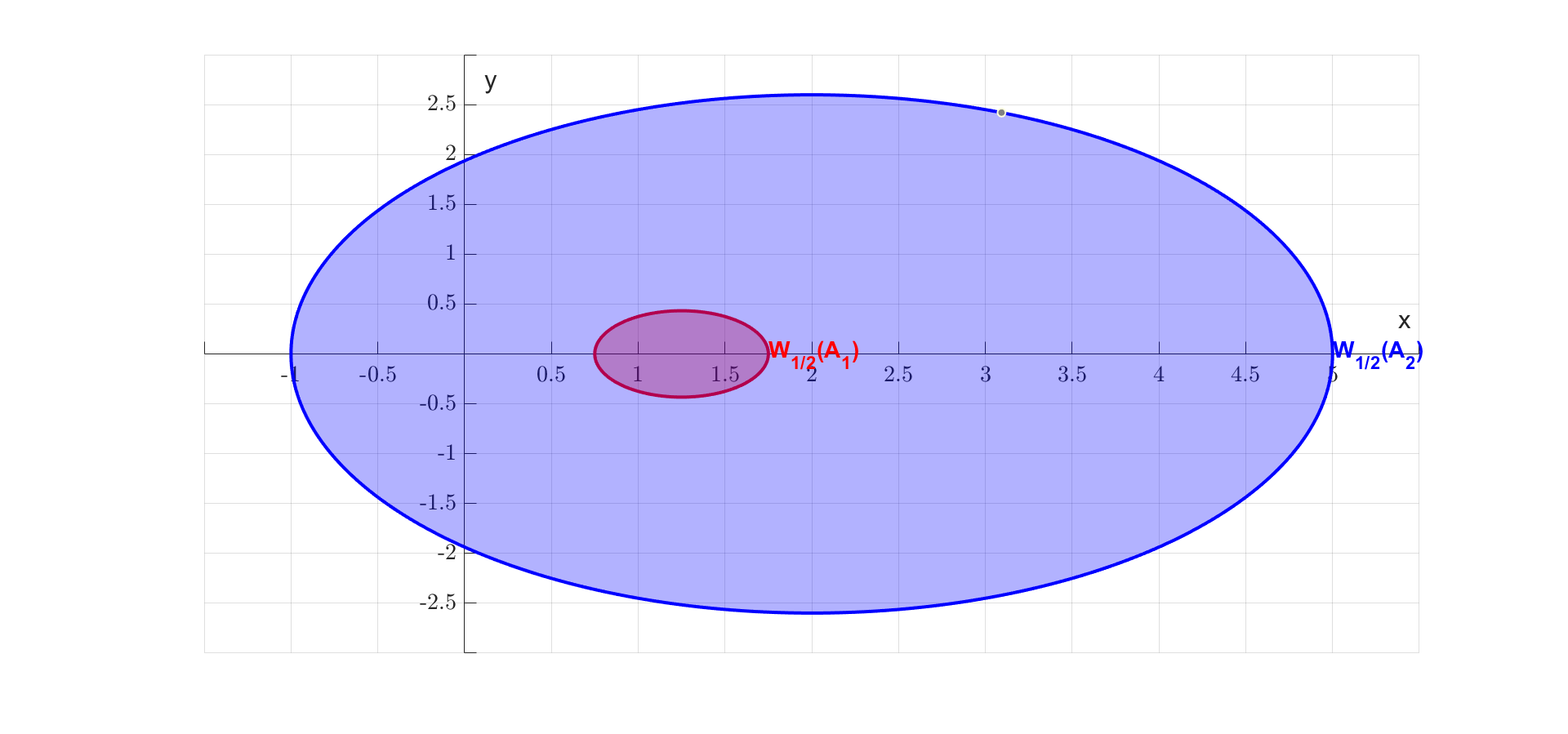} 
    \caption{Shapes of $W_{\frac{1}{2}}(A_1)$ and $W_{\frac{1}{2}}(A_2)$.}   
\end{center}    
\end{figure}
  
The above figure illustrates the existence of $\alpha \in [0,\frac{\pi}{2})$ such that $W_\frac{1}{2}(A_1) \in S_\alpha$, while no $\alpha \in [0,\frac{\pi}{2})$ satisfies $W_\frac{1}{2}(A_2) \in S_\alpha$. Consequently, $A_2$ is a positive semidefinite matrix; however, it is not a $q$-sectorial matrix. Therefore, positive matrices may not $q$-sectorial matrices.
\end{example}

There are certain well-known classes of matrices, which cannot be classified as $q$-sectorial matrices as follows:
\begin{itemize}
    \item [(i)] If all the entries of a matrix $A$ are purely imaginary and $q \in (0,1)$, then $W_q(A)^*=W_q(\overline{A})=-W_q(A)$ (Lemma \ref{nlma}). Thus, $W_q(A)$ is symmetric about $y$-axis. 
    \item[(ii)] If $q=0$, then $W_0(A)$ is a circular disc centered at origin \cite{stolov1979hausdorff}. 
 \item[(iii)] If $A$ is a finite-dimensional shift operator and $q \in \overline{\mathcal{D}}$, then  $W_q(A)$ is a closed disc centered at the origin \cite{citation-0},
 \item[(iv)] If $A$ is a nonzero square-zero matrix and and $q \in \overline{\mathcal{D}}$, then  $W_q(A)$ is a closed disc centered at the origin \cite{rani2025q}.
\end{itemize}
Thus, from now onward, we assume $q \ne 0.$
To prove one of our main results, the following lemma is required.
\begin{lemma}\label{lemmaimag}
Let $A \in \prod_{q,\alpha}^n$, $q \in \overline{\mathcal{D}}$ and $\phi$ be an Orlicz function. Then 
\begin{equation*}
	\phi\left(|\Im \langle Ax,y \rangle|\right) \le \sin(\alpha)\phi \left(w_q(A)\right),
\end{equation*}
where $x,y \in \mathcal{H}$ with $\|x\|=\|y\|=1$ and $\langle x,y \rangle=q$.
\end{lemma} 
\begin{proof}
Consider
	\begin{align*}
		|\langle Ax,y \rangle|&=\sqrt{(\Re\langle Ax,y \rangle)^2+(\Im \langle Ax,y \rangle)^2}\\
		& \ge \sqrt{\cot^2(\alpha)(\Im \langle Ax,y \rangle)^2+(\Im \langle Ax,y \rangle)^2}\\
		&=\sqrt{(\cot^2(\alpha)+1)(\Im \langle Ax,y \rangle)^2}\\
		&=\csc(\alpha)|\Im \langle Ax,y \rangle|.
	\end{align*} 
This implies
\begin{equation*}
	|\Im \langle Ax,y \rangle| \le \sin(\alpha)|\langle Ax,y \rangle|\le \sin(\alpha)w_q(A).
\end{equation*}
Using $\phi(\mu x) \le \mu \phi(x)$ where $0 \le \mu \le 1$, we have
\begin{equation*}
	\phi\left(|\Im \langle Ax,y \rangle|\right) \le \sin(\alpha)\phi \left(w_q(A)\right).
\end{equation*}


which completes the proof.
\end{proof}
Using Lemma \ref{lemmaimag}, the following theorem establishes the $q$-sectorial analog of Lemma \ref{img} and Lemma \ref{kitsec}.
\begin{theorem}\label{s2t4.8}
Let $A \in \prod_{q,\alpha}^n$, $q \in \overline{\mathcal{D}}\setminus \{0\}$ and $\phi$ be an Orlicz function. Then
\begin{itemize}
\item[(i)]   
$
\|\Im(A)\| \le \left(\frac{1}{|q|}\sin(\alpha)+\frac{2\sqrt{1-|q|^2}}{|q|^2}\right)w_q(A),
$

\item[(ii)]
  \begin{align*}
\phi \left(\frac{\|AA^*+A^*A\|}{2}\right)\le & \int_0^1 \phi\left( \frac{2t}{|q|^2}w_q^2(A)+2(1-t) \left(\frac{\sin(\alpha)}{|q|}+\frac{2\sqrt{1-|q|^2}}{|q|^2}\right)^2w_q^2(A) \right) dt\\
\le& \frac{1}{2}\left(\phi\left( \frac{2}{|q|^2}w_q^2(A)\right)+ \phi \left( 2\left(\frac{\sin(\alpha)}{|q|}+\frac{2\sqrt{1-|q|^2}}{|q|^2}\right)^2w_q^2(A) \right)\right).
\end{align*}
\end{itemize}
\end{theorem}
\begin{proof}
\begin{itemize}
    \item [(i)] 
If $\phi(t)=t$ in Lemma \ref{lemmaimag}, we have
\begin{equation*}
    |\Im \langle Ax,y \rangle| \le \sin(\alpha)w_q(A).
\end{equation*}
Taking $y=\overline{q}x+\sqrt{1-|q|^2}z$, where $\|z\|=1$ and $\langle x,z \rangle =0$, we obtain 
\begin{eqnarray*}
		|\Im \langle Ax,\overline{q}x+\sqrt{1-|q|^2}z \rangle| \le \sin(\alpha)w_q(A)\\
		|\overline{q} \Im \langle Ax,x \rangle+\sqrt{1-|q|^2}\Im \langle Ax,z \rangle|\le \sin(\alpha)w_q(A).
\end{eqnarray*}        
\begin{eqnarray*}
  |\overline{q} \Im \langle Ax,x \rangle|-\sqrt{1-|q|^2}|\Im \langle Ax,z \rangle|\le \sin(\alpha)w_q(A) \nonumber\\
 |q| |\Im\langle Ax,x \rangle| \le  \sin(\alpha)w_q(A)+\sqrt{1-|q|^2}\|A\|.
\end{eqnarray*}
Using  Lemma \ref{t1.15}, we have
\begin{align*}
	|q| |\Im \langle Ax,x \rangle|\le & \sin(\alpha)w_q(A)+\frac{2\sqrt{1-|q|^2}}{|q|}w_q(A).
    \end{align*}
Taking the supremum over $x$ with $\|x\|=1$, we have
\begin{align*}
\|\Im(A)\| \le& \left(\frac{1}{|q|}\sin(\alpha)+\frac{2\sqrt{1-|q|^2}}{|q|^2}\right)w_q(A).
\end{align*}
\item[(ii)] 
	We have
	\begin{align*}
		\frac{\|AA^*+A^*A\|}{2} = &\|	\Re^2(A)+\Im^2(A)\|\\
		\le & \left(\|	\Re(A)\|^2+\|	\Im(A)\|^2\right).
	\end{align*}
Therefore,
\[ \phi \left(\frac{\|AA^*+A^*A\|}{2}\right)\le \phi \left(\frac{2\|	\Re(A)\|^2+2\|	\Im(A)\|^2}{2}\right).\]
Using the Hermite-Hadamard inequality, we obtain that
\begin{align*}
\phi \left(\frac{\|AA^*+A^*A\|}{2}\right)\le \int_0^1 \phi\left( 2t\|	\Re(A)\|^2+2(1-t) \|	\Im(A)\|^2 \right) dt.   
\end{align*}
 For a normal matrix $A$, we have $w(A)=\|A\|$ \cite[p.97]{gau2021numerical}. Thus, Lemma \ref{t1.15} implies that $|q|\|\Re(A)\| \leq w_q(\Re(A)) \leq w_q(A)$, and hence $\|\Re(A)\|\le \frac{1}{|q|}w_q(A)$. By using this relation along with part (i), we obtain
 \begin{align*}
\phi \left(\frac{\|AA^*+A^*A\|}{2}\right)\le & \int_0^1 \phi\left( \frac{2t}{|q|^2}w_q^2(A)+2(1-t) \left(\frac{\sin(\alpha)}{|q|}+\frac{2\sqrt{1-|q|^2}}{|q|^2}\right)^2w_q^2(A) \right) dt\\
\le & \int_0^1 \left(t\phi\left( \frac{2}{|q|^2}w_q^2(A)\right)+(1-t) \phi \left( 2\left(\frac{\sin(\alpha)}{|q|}+\frac{2\sqrt{1-|q|^2}}{|q|^2}\right)^2w_q^2(A) \right)\right) dt\\
=& \frac{1}{2}\left(\phi\left( \frac{2}{|q|^2}w_q^2(A)\right)+ \phi \left( 2\left(\frac{\sin(\alpha)}{|q|}+\frac{2\sqrt{1-|q|^2}}{|q|^2}\right)^2w_q^2(A) \right)\right).
\end{align*}
Thus, the desired result follows.
\end{itemize}
\end{proof}
\begin{remark}\label{remark4.2}
If $\phi(t)=t^r,~ r \ge 1$, then Theorem \ref{s2t4.8} gives
\begin{align*}
 \left(\frac{\|AA^*+A^*A\|}{2}\right)^r\le & \int_0^1 \left( \frac{2t}{|q|^2}w_q^2(A)+2(1-t) \left(\frac{\sin(\alpha)}{|q|}+\frac{2\sqrt{1-|q|^2}}{|q|^2}\right)^2w_q^2(A) \right)^r dt\\
\le& 2^{r-1}\left( \frac{1}{|q|^{2r}}w_q^{2r}(A)+   \left(\frac{\sin(\alpha)}{|q|}+\frac{2\sqrt{1-|q|^2}}{|q|^{2}}\right)^{2r}w_q^{2r}(A) \right).
\end{align*}
\end{remark}
\begin{corollary}\label{corimag}
Let $A \in \prod_{q,\alpha}^n$ and $q \in \overline{\mathcal{D}}\setminus \{0\}$. Then
\begin{equation*}
    \frac{|q|^4\|AA^*+A^*A\|}{2\left({|q|^2}+\left(|q|{\sin(\alpha)}+2\sqrt{1-|q|^2}\right)^2\right)}\le w_q^2(A).
\end{equation*}
\end{corollary}
\begin{proof}
Setting $r=1$ in Remark \ref{remark4.2}, we have
\begin{align*}
\frac{\|AA^*+A^*A\|}{2}\le & \int_0^1 \left( \frac{2t}{|q|^2}w_q^2(A)+2(1-t) \left(\frac{\sin(\alpha)}{q}+\frac{2\sqrt{1-|q|^2}}{|q|^2}\right)^2w_q^2(A)\right)  dt\\
=&  \frac{1}{|q|^2}w_q^2(A)+\left(\frac{\sin(\alpha)}{|q|}+\frac{2\sqrt{1-|q|^2}}{|q|^2}\right)^2w_q^2(A).
\end{align*}
Therefore,
\begin{align*}
		\frac{\|AA^*+A^*A\|}{2}\le &  \left(\frac{1}{|q|^2}+\left(\frac{\sin(\alpha)}{|q|}+\frac{2\sqrt{1-|q|^2}}{|q|^2}\right)^2\right)w_q^2(A).\\
		 = & \frac{1}{|q|^4} \left({|q|^2}+\left(|q|{\sin(\alpha)}+2\sqrt{1-|q|^2}\right)^2\right)w_q^2(A).
\end{align*}
\end{proof}
Corollary \ref{corimag} reduces to Lemma \ref{kitsec} in the special case when $q=1$.


Our next theorem gives a complete refinement of inequality \eqref{eq2.1} in the case of sectorial matrices. Moreover, it refines Lemma \ref{img} and Lemma \ref{kitsec} under certain conditions.
\begin{theorem}\label{mt4.3} 
	Let $A \in \prod_{q,\alpha}^n$ and $q \in \overline{\mathcal{D}}\setminus\{0\}$. Then
\begin{itemize}

\item [(i)] $ \|A\| \le \left(\frac{|q|+|q|\sin\alpha +{2\sqrt{1-|q|^2}}}{|q|^2}\right)w_q(A),$
\item [(ii)]  $
    \left(\frac{|q|^2}{|q|\sin(\alpha)+2\sqrt{1-|q|^2}}\right)^2\left(\frac{\|A^*A+AA^*\|}{4}+
   \frac{ \|\Im(A)\|^2-\|\Re(A)\|^2}{2}\right) \le w_q^2(A),
$
\item[(iii)] $
   \left(\frac{q^2}{q\sin(\alpha)+2\sqrt{1-q^2}}\right)\left(\frac{\|A\|}{2}+
   \frac{ \|\Im(A)\|-\|\Re(A)\|}{2}\right) \le  w_q(A).
$
\end{itemize}    	
\end{theorem}
\begin{proof}
\begin{itemize}

\item[(i)] 
By the Cartesian decomposition of $A$, we have
	\begin{equation*}
		\|A\| \le	\|\Re(A)\|+	\|\Im(A)\|.  
	\end{equation*}	
 Using the relation $\|\Re(A)\|\le \frac{1}{|q|}w_q(A)$ and Theorem \ref{remark4.2}(i), we obtain
	\begin{align*}
		\|A\| \le & \frac{1}{|q|}w_q(A)+\left(\frac{1}{|q|}\sin(\alpha)+\frac{2\sqrt{1-|q|^2}}{|q|^2}\right)w_q(A)\\
   =& \left(\frac{|q|+|q|\sin\alpha +{2\sqrt{1-|q|^2}}}{|q|^2}\right)w_q(A).
 \end{align*}

 \item [(ii)]From Theorem \ref{remark4.2}(i), we have
\begin{align*}
    w_q^{2}(A) \ge &\left(\frac{|q|^2}{|q|\sin(\alpha)+2\sqrt{1-|q|^2}}\right)^{2} \|\Im(A)\|^{2}\\
   = &\frac{1}{2}\left(\frac{|q|^2}{|q|\sin(\alpha)+2\sqrt{1-|q|^2}}\right)^{2}\left(\|\Re(A)\|^{2}+ \|\Im(A)\|^{2}\right)\\
   +&\frac{1}{2}\left(\frac{|q|^2}{|q|\sin(\alpha)+2\sqrt{1-|q|^2}}\right)^{2}\left( \|\Im(A)\|^{2}-\|\Re(A)\|^{2}\right)\\
    \ge &\frac{1}{2}\left(\frac{|q|^2}{|q|\sin(\alpha)+2\sqrt{1-|q|^2}}\right)^{2}\left(\|\Re^{2}(A)+\Im^{2}(A)\|\right)\\
   +&\frac{1}{2}\left(\frac{|q|^2}{|q|\sin(\alpha)+2\sqrt{1-|q|^2}}\right)^{2}\left( \|\Im(A)\|^{2}-\|\Re(A)\|^{2}\right)\\
   = &\frac{1}{4}\left(\frac{|q|^2}{|q|\sin(\alpha)+2\sqrt{1-|q|^2}}\right)^2\|A^*A+AA^*\| \nonumber\\
   +&\frac{1}{2}\left(\frac{|q|^2}{|q|\sin(\alpha)+2\sqrt{1-|q|^2}}\right)^2\left( \|\Im(A)\|^2-\|\Re(A)\|^2\right).
\end{align*}
\item[(iii)] From Theorem \ref{remark4.2}(i), we have
\begin{align*}
    w_q(A)
   \ge &\frac{1}{2}\left(\frac{|q|^2}{|q|\sin(\alpha)+2\sqrt{1-|q|^2}}\right)\left(\|\Re(A)\|+ \|\Im(A)\|\right)\\
   +&\frac{1}{2}\left(\frac{|q|^2}{|q|\sin(\alpha)+2\sqrt{1-|q|^2}}\right)\left( \|\Im(A)\|-\|\Re(A)\|\right).
   \end{align*}
\end{itemize}
As similar to part (ii), we can obtain the desired result easily.
This completes the proof.
\end{proof}
 
By taking $q=1$ in Theorem \ref{mt4.3}(i), we can state the following corollary.
 \begin{corollary}\label{re3.2eq}
 Let $A \in \prod_{1,\alpha}^n$. Then
\begin{equation*}
        \frac{1}{1+\sin(\alpha)}\|A\|\le w(A) \le \|A\|.
    \end{equation*}     
 \end{corollary}
 \begin{remark}
    Clearly, the lower bound stated in Corollary \ref{re3.2eq} is a refinement of the lower bound mentioned in inequalities (\ref{eq2.1}) for sectorial matrices. 
  Also, for a certain range of $\alpha$, which is $\frac{6}{19}\pi < \alpha < \frac{\pi}{2}$, Corollary \ref{re3.2eq} gives a better lower bound of $w(A)$ than the bound $\cos \alpha \|A\|\le w(A) \le \|A\|$ (Lemma \ref{ip1.5}).
In the following figure, a comparison of the lower bounds for $w(A)$, as stated in Lemma \ref{ip1.5} and Corollary \ref{re3.2eq}, is presented.  


   \begin{figure}[H]
\begin{center}
 \includegraphics[scale=0.2]{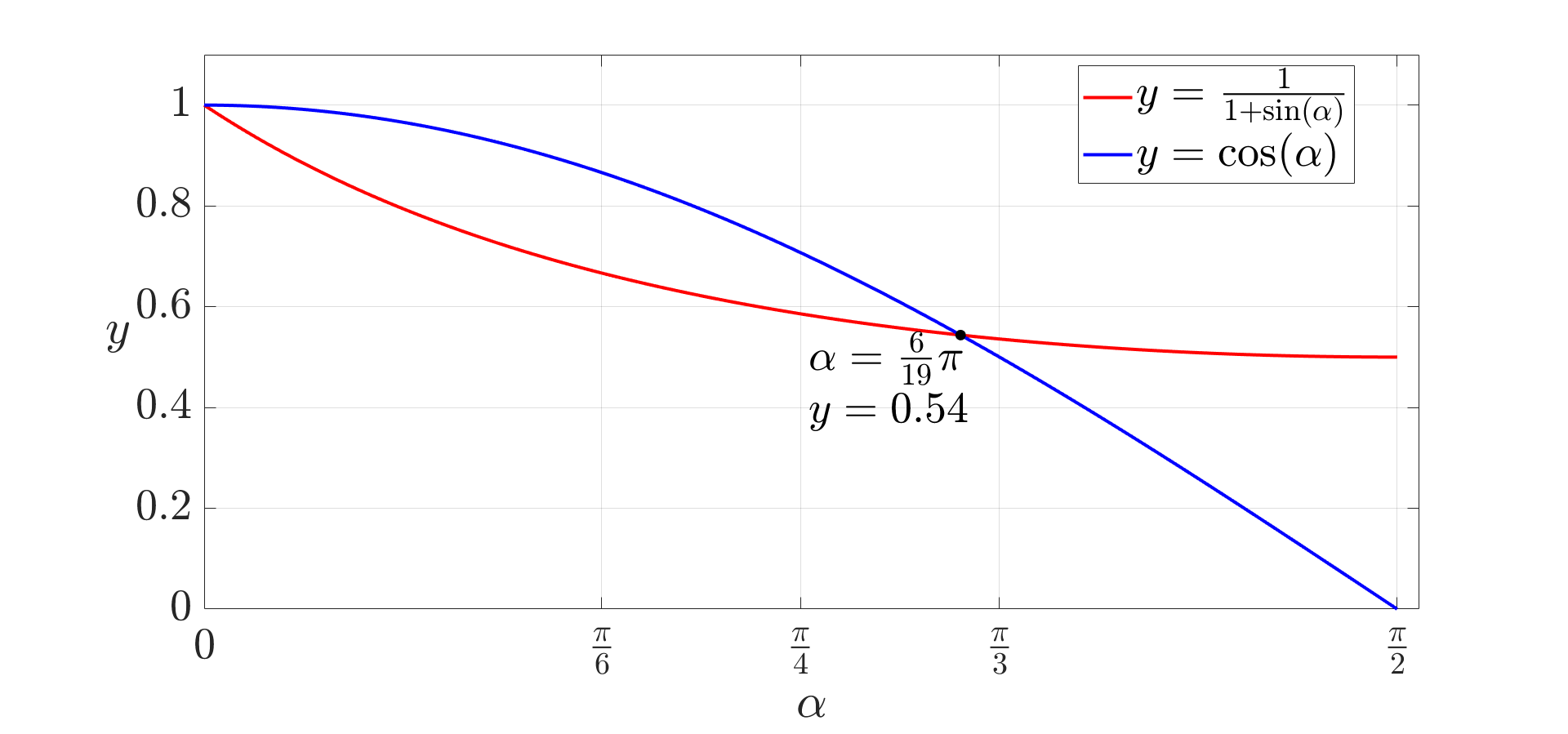} 
    \caption{Comparison $\cos(\alpha)$ and $\frac{1}{1+\sin(\alpha)}$}   
\end{center}    
\end{figure}  
\end{remark}

\begin{remark}
 For $q=1$, part (ii) and part (iii) of Theorem \ref{mt4.3} yield the results in \cite[Theorem 2.2]{bhunia2023numerical} and  \cite[Theorem 3.5]{sammour2022geometric}, respectively.  
\end{remark}

If $q=1$ and $\|\Re(A)\| \le \|\Im(A)\|$ in Theorem \ref{mt4.3}(ii), we have the following corollary:
\begin{corollary}\label{newc}
  Let $A \in \prod_{1,\alpha}^n$ and $\alpha \ne 0$. Then 
  \begin{equation*}
    \frac{1}{4\sin^2(\alpha)}\|A^*A+AA^*\| \le w^2(A).
\end{equation*}
\end{corollary}

Clearly, if $\|\Re(A)\| \le \|\Im(A)\|$, Corollary \ref{newc} is a refinement of Lemma \ref{kitsec}. Since $\|T\|^2 \le \|TT^*+T^*T\|$, it follows that the aforementioned corollary is a refinement of the left inequality in Lemma \ref{ip1.5} and Corollary \ref{re3.2eq} when $\|\Re(A)\| \le \|\Im(A)\|$. 

The forthcoming remark compares the bounds given in Corollary \ref{corimag} and Theorem \ref{mt4.3}(ii).
\begin{remark}
 Let $A=\begin{pmatrix}
     0.4889+0.29239i & 0.7269+0.8884i\\
     1.0347-0.7873i & -0.3034-1.1471i
 \end{pmatrix}$, which is randomly generated by MATLAB command \texttt{randn} and $q \in (0,1)$. Then we have $\frac{\|AA^*+A^*A\|}{2}=3.4433$ and $\frac{\|AA^*+A^*A\|}{4}+\frac{\|\Im(A)\|^2-\|\Re(A)\|^2}{2}=2.1623$. It follows from Corollary \ref{corimag}, 
 \begin{equation}\label{bd1}
    w_q^2\left(\begin{pmatrix}
     0.4889+0.29239i & 0.7269+0.8884i\\
     1.0347-0.7873i & -0.3034-1.1471i
 \end{pmatrix} \right)\ge\frac{3.4433q^4}{q^2+(q\sin\alpha +2\sqrt{1-q^2})^2}.
\end{equation}
Also, from Theorem \ref{mt4.3}(ii), we have
\begin{equation}\label{bd2}
    w_q^2\left(\begin{pmatrix}
     0.4889+0.29239i & 0.7269+0.8884i\\
     1.0347-0.7873i & -0.3034-1.1471i
 \end{pmatrix} \right)\ge\frac{2.1623q^4}{(q\sin \alpha +2 \sqrt{1-q^2)^2}}.
\end{equation}
In the following figures, we compare the aforementioned bounds in \eqref{bd1} and \eqref{bd2} for particular choices of $\sin \alpha$.
\begin{figure}[H]
\begin{minipage}[b]{0.5\textwidth}
 \includegraphics[width=\linewidth]{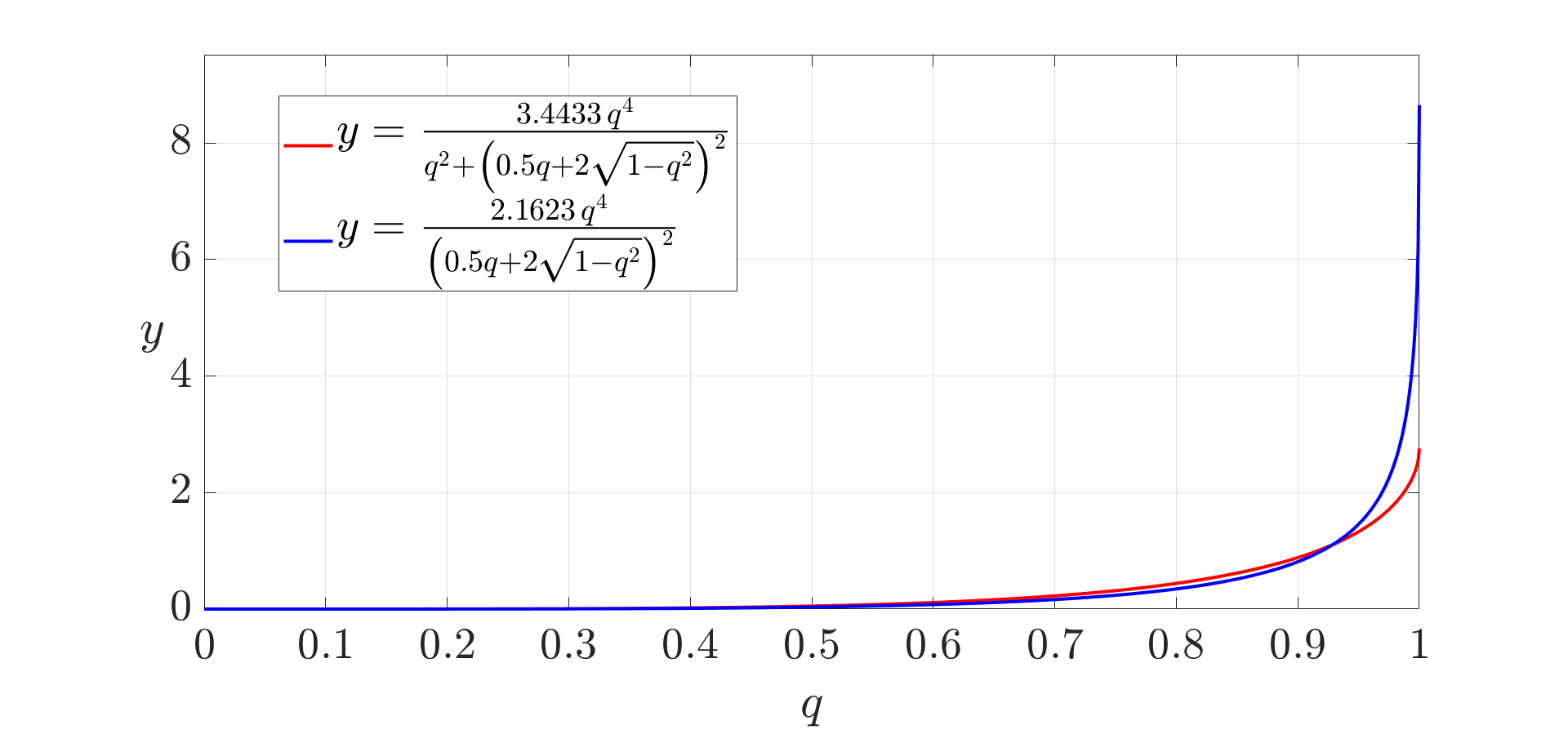}
 \caption{If $\sin(\alpha)=0.5.$}
\end{minipage}\hfill
\begin{minipage}[b]{0.5\textwidth}
  \includegraphics[width=\linewidth]{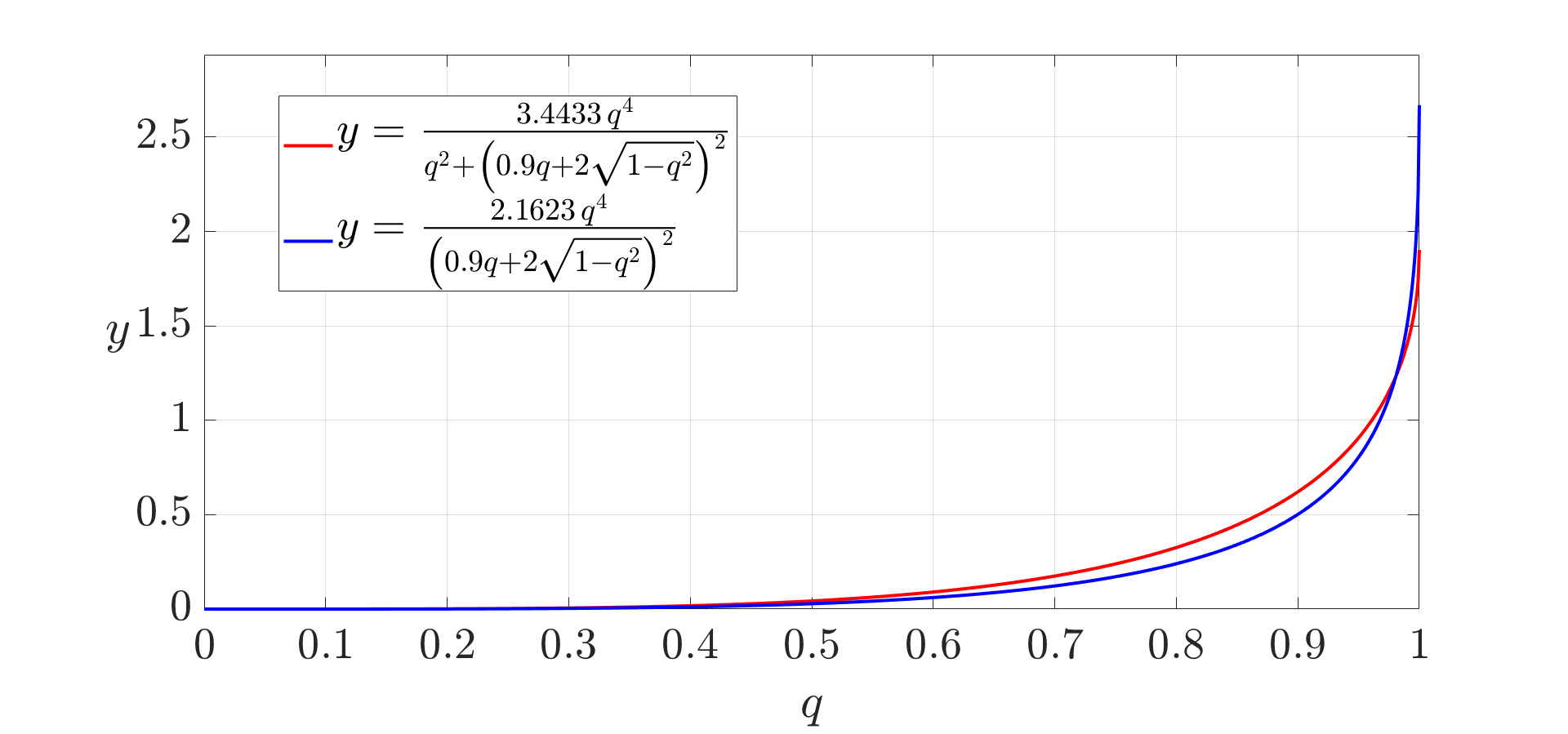}
  \caption{If $\sin(\alpha)=0.9.$}
\end{minipage}\hfill
\end{figure}
From the above figures, it is evident that the bounds for $w_q(A)$ given in Corollary \ref{corimag} and Theorem \ref{mt4.3}(ii) are not comparable.
\end{remark}

At last, we derive some upper bounds for the $q$-numerical radii of generalized commutator and anti-commutator matrices. Recall that for $A, B \in M_n$, the matrix $AB - BA$ is called the commutator matrix and $AB + BA$ is called the anti-commutator matrix.

\begin{theorem}\label{th3.6}
Let $A, B, X, Y \in M_n$ with $A \in \prod_{q,\alpha}^n$, $q \in \overline{\mathcal{D}}\setminus \{0\}$ and $\phi$ be an Orlicz function. Then
\begin{align*}
&w_q(AXB \pm BYA) \\
&\leq ~2|q| \max \{ \| XB \|, \| BY \| \} \Bigg( \phi \left(\frac{1}{|q|}w_q(A)\right)+\psi \left(\frac{1}{|q|}w_q(A)\right)+\phi \left(\left( \frac{\sin(\alpha)}{|q|}+\frac{2\sqrt{1-|q|^2}}{|q|^2}\right)w_q(A) \right)\\
&+~ \psi \left( \left(\frac{\sin(\alpha)}{|q|}+\frac{2\sqrt{1-|q|^2}}{|q|^2}\right)w_q(A) \right)\Bigg)^\frac{1}{2}
+\sqrt{1-|q|^2}\|AXB \pm BYA\|.
\end{align*}
where $\psi$ is a complementary Orlicz function to $\phi$.
\end{theorem}
\begin{proof}
Theorem \ref{th3.6} holds trivially when $X = Y = 0$. Let $(X,Y) \ne (0,0)$ and $x,y \in \mathbb{C}^n$ with $\| x \|=\|y\| = 1,$ $\langle x,y \rangle=q$. 
Setting $X'=\frac{X}{\max \{ \| X \|, \| Y \| \}}$, $Y'=\frac{Y}{\max \{ \| X \|, \| Y \| \}}$ and $y=\overline{q}x+\sqrt{1-|q|^2}z$, where $\|z\|=1$ and $\langle x,z \rangle =0$,  we obtain
\begin{align*}
 | \langle (AX' \pm Y'A) x, y \rangle | =&   | \langle (AX' \pm Y'A) x, \overline{q}x+\sqrt{1-|q|^2}z \rangle |\\
 \le & |q||\langle (AX' \pm Y'A) x,x \rangle |+\sqrt{1-|q|^2}|\langle (AX' \pm Y'A) x,z \rangle |\\
 \le& |q|| \langle AX' x, x \rangle \pm \langle Y'A x, x \rangle |+\sqrt{1-|q|^2}\|AX' \pm Y'A\| \nonumber \\
\le& |q|\left(| \langle X' x, A^*x \rangle | + | \langle A x, Y'^*x \rangle |\right) +\sqrt{1-|q|^2}\|AX' \pm Y'A\| \nonumber \\
\leq & |q|\left(\|A^*x\| + \|Ax\| \right)+\sqrt{1-|q|^2}\|AX' \pm Y'A\|\nonumber \\
\leq & |q|\sqrt{2( \| A^* x\|^2 + \|A x\|^2 )} +\sqrt{1-|q|^2}\|AX' \pm Y'A\|\\
\leq & |q|\sqrt{2\|AA^*+A^*A\|}+\sqrt{1-|q|^2}\|AX' \pm Y'A\|\nonumber. 
\end{align*}
By using the triangle inequality, we have
	\begin{align*}
		\|AA^*+A^*A\| = &2\|	\Re^2(A)+\Im^2(A)\|
		\le 2 \left(\|	\Re(A)\|^2+\|	\Im(A)\|^2\right)\\
        = & 2 \left( \phi \left(\|\Re(A)\|\right)+\psi \left(\|\Re(A)\|\right)+\phi \left(\|\Im(A)\|\right)+\psi \left(\|\Im(A)\|\right) \right) \quad (\text{by Young's inequality).}
	\end{align*}
By using the relation $\|\Re(A)\|\le \frac{1}{|q|}w_q(A)$ and Theorem \ref{s2t4.8}(i), we obtain that
\begin{align*}
 \|AA^*+A^*A\|
\le &~ 2\Bigg( \phi \left(\frac{1}{|q|}w_q(A)\right)+\psi \left(\frac{1}{|q|}w_q(A)\right)\\
+&~\phi\left( \left( \frac{\sin(\alpha)}{|q|}
+\frac{2\sqrt{1-|q|^2}}{|q|^2}\right)w_q(A) \right)+\psi\left( \left(\frac{\sin(\alpha)}{|q|}+\frac{2\sqrt{1-|q|^2}}{|q|^2}\right)w_q(A)\right) \Bigg).   
\end{align*}
Hence,
\begin{align*}
| \langle (AX' + Y'A) x, y \rangle | \leq &~2|q|\Bigg( \phi \left(\frac{1}{|q|}w_q(A)\right)+\psi \left(\frac{1}{|q|}w_q(A)\right)+\phi \left(\left( \frac{\sin(\alpha)}{|q|}+\frac{2\sqrt{1-|q|^2}}{|q|^2}\right)w_q(A) \right)\\
+&~\psi\left( \left(\frac{\sin(\alpha)}{|q|}+\frac{2\sqrt{1-|q|^2}}{|q|^2}\right) w_q(A)\right)\Bigg)^\frac{1}{2}
+\sqrt{1-|q|^2}\|AX' \pm Y'A\|.
\end{align*}

Taking the supremum over $x,y \in \mathbb{C}^n, \| x \| =\|y\|= 1$ and $\langle x,y \rangle=q$, we have
\begin{align*}\label{in10}
  w_q(AX' \pm Y'A)
 \leq &~2|q|\Bigg( \phi \left(\frac{1}{|q|}w_q(A)\right)+\psi \left(\frac{1}{|q|}w_q(A)\right)+\phi \left(\left( \frac{\sin(\alpha)}{|q|}+\frac{2\sqrt{1-|q|^2}}{|q|^2}\right)w_q(A) \right)\\
+&~\psi\left( \left(\frac{\sin(\alpha)}{|q|}+\frac{2\sqrt{1-|q|^2}}{|q|^2}\right) w_q(A)\right)\Bigg)^\frac{1}{2}
+\sqrt{1-|q|^2}\|AX' \pm Y'A\|.  
\end{align*}

It also follows that
\begin{align*}
&w_q(AX \pm YA) \\
 \leq &~ 2|q|\max \{ \| X \|, \| Y \| \} \Bigg( \phi \left(\frac{1}{|q|}w_q(A)\right)+\psi \left(\frac{1}{|q|}w_q(A)\right)+\phi\left( \left( \frac{\sin(\alpha)}{|q|}+\frac{2\sqrt{1-|q|^2}}{|q|^2}\right)w_q(A) \right)\\
 +&~\psi \left( \left(\frac{\sin(\alpha)}{|q|}
+\frac{2\sqrt{1-|q|^2}}{|q|^2}\right)w_q(A) \right) \Bigg)^\frac{1}{2}
+\sqrt{1-|q|^2}\|AX \pm YA\|.
\end{align*}

Replacing $X$ and $Y$ by $XB$ and $BY$, respectively, we get
\begin{align*}
&w_q(AXB \pm BYA) \\
\leq & ~2|q| \max \{ \| XB \|, \| BY \| \} \Bigg( \phi \left(\frac{1}{|q|}w_q(A)\right)+\psi \left(\frac{1}{|q|}w_q(A)\right)+\phi \left(\left( \frac{\sin(\alpha)}{|q|}+\frac{2\sqrt{1-|q|^2}}{|q|^2}\right)w_q(A) \right)\\
+&~\psi \left( \left(\frac{\sin(\alpha)}{|q|}+\frac{2\sqrt{1-|q|^2}}{|q|^2}\right) w_q(A) \right)\Bigg)^\frac{1}{2}
+\sqrt{1-|q|^2}\|AXB \pm BYA\|.
\end{align*}
This completes the  proof.
\end{proof}
\begin{corollary}\label{aliin}
Let $A, B, X, Y \in M_n$ with $A \in \prod_{q,\alpha}^n$ and $q \in \overline{\mathcal{D}}\setminus \{0\}$. Then  
  \begin{align*}
&w_q(AXB \pm BYA) \nonumber\\
\leq& \frac{2}{|q|}\max \{ \| XB \|, \| BY \| \} \left({|q|^2+\left(|q| \sin(\alpha)+2 \sqrt{1-|q|^2}\right)^2}\right)^\frac{1}{2}w_q(A)+\sqrt{1-|q|^2}\|AXB \pm BYA\|.
\end{align*}
\end{corollary}
\begin{proof}
 Taking $\phi(t)=\psi(t)=\frac{t^2}{2}$ in Theorem \ref{th3.6}, the required result is obvious.   
\end{proof}
\begin{corollary}\label{incor}
Let $A,B\in M_n$ with $A \in \prod_{q,\alpha}^n$ and $q \in \overline{\mathcal{D}}\setminus \{0\}$. Then
   \begin{equation*}
        w_q(AB \pm BA) \leq  \frac{2}{|q|}\| B\|\left({|q|^2+\left(|q| \sin(\alpha)+2 \sqrt{1-|q|^2}\right)^2}\right)^\frac{1}{2}w_q(A)+\sqrt{1-|q|^2}\|AB \pm BA\|.
   \end{equation*}
\end{corollary}
\begin{proof}
    The required result can be easily obtained by taking $X=Y=I$ in Corollary \ref{aliin}.
\end{proof}
\begin{corollary}
 Let $A,B \in \prod_{q,\alpha}^n$ and $q \in \overline{\mathcal{D}}\setminus \{0\}$. Then
 \begin{equation*}
     w_q(AB \pm BA) \le \min\{ \mu_1, \mu_2\},
 \end{equation*}
  where
 \begin{equation*}
\mu_1=\frac{2}{|q|}\| B\|\left({|q|^2+\left(|q| \sin(\alpha)+2 \sqrt{1-|q|^2}\right)^2}\right)^\frac{1}{2}w_q(A)+\sqrt{1-|q|^2}\|AB \pm BA\|     
 \end{equation*} and
 \begin{equation*}
     \mu_2=\frac{2}{|q|}\| A\|\left({|q|^2+\left(|q| \sin(\alpha)+2 \sqrt{1-|q|^2}\right)^2}\right)^\frac{1}{2}w_q(B)+\sqrt{1-q^2}\|AB \pm BA\|.
 \end{equation*}
 \end{corollary}
 \begin{proof}
     The proof follows from Corollary \ref{incor} by interchanging $A$ and $B$.
 \end{proof}

 If we take $q=1$ in Corollary \ref{incor}, then we have the following corollary.
 \begin{corollary}
 Let $A,B\in M_n$ with $A \in \prod_{1,\alpha}^n$. Then
   \begin{equation}\label{inre}
  w(AB \pm BA) \le 2\sqrt{1+\sin^2(\alpha)}\| B\|w(A). 
\end{equation}  
 \end{corollary}
 
\begin{remark}
As $\sin(\alpha) \le 1$ for $\alpha \in [0, \frac{\pi}{2})$, it follows that for $A \in \prod_{q,\alpha}^n$, inequality (\ref{inre}) gives a sharper bound in comparison to the well-known bound \cite[Theorem 11]{fong1983unitarily}, namely
\begin{equation*}
 w(AB \pm BA) \le 2\sqrt{2}\| B\|w(A).   
\end{equation*}
\end{remark}


\section*{Declarations}

\textbf{Funding} 

Not applicable.
\vspace{0.5cm}

{\bf Availability of data and materials}

No data were used to support this study.
\vspace{0.5cm}

\textbf{Competing interests}

The authors declare that they have no conflict of interest.

\vspace{0.5cm}

\textbf{Author's contribution}

All authors contributed equally to this work and read and approved the final manuscript.
\bibliographystyle{plain}
\bibliography{bib_NR_SP}
 \end{document}